\documentclass[reqno]{amsart}
\usepackage{amsmath,amssymb,latexsym, amsfonts, amscd, amsthm}
\usepackage{verbatim}

\usepackage{xcolor}
\usepackage{hyperref}

\hypersetup{colorlinks=false,linkbordercolor=red,linkcolor=green,pdfborderstyle={/S/U/W 1}}

\usepackage{bbm}
\usepackage{enumerate}
\usepackage[all]{xy}

\usepackage{color} 

\numberwithin{equation}{section} 

\allowdisplaybreaks
\newtheorem{thm}{Theorem}[section]
\newtheorem{pro}[thm]{Proposition}
\newtheorem{lm}[thm]{Lemma}

\numberwithin{equation}{section}

\theoremstyle{remark}

\newtheorem{rem}[thm]{Remark}

\theoremstyle{definition}

\DeclareMathOperator*{\End}{End}

\DeclareMathOperator*{\Irr}{Irr}
\DeclareMathOperator*{\disc}{disc}
\DeclareMathOperator*{\Gal}{Gal}

\DeclareMathOperator*{\Sp}{Sp}

\DeclareMathOperator*{\SU}{SU}

\DeclareMathOperator*{\U}{U}
\DeclareMathOperator*{\SO}{SO}
\DeclareMathOperator*{\PSO}{PSO}
\newcommand{\SL}{\mathrm{SL}} 	
\newcommand{\GL}{\mathrm{GL}} 

\DeclareMathOperator*{\GSO}{GSO}
\DeclareMathOperator*{\GSp}{GSp}
\newcommand{\Spin}{\mathrm{Spin}}	
\DeclareMathOperator*{\PSL}{PSL}
\DeclareMathOperator*{\PGL}{PGL}
\newcommand{\GSpin}{\mathrm{GSpin}}

\DeclareMathOperator*{\Ind}{Ind}
\DeclareMathOperator*{\simi}{sim}
\DeclareMathOperator*{\sgn}{sgn}

\DeclareMathOperator*{\temp}{temp}
\DeclareMathOperator*{\scup}{scusp}  

\DeclareMathOperator*{\unit}{unit}
\DeclareMathOperator*{\ad}{ad}
\DeclareMathOperator*{\scn}{sc}
\DeclareMathOperator*{\el}{ell}

\DeclareMathOperator*{\Nrd}{Nrd}
\DeclareMathOperator*{\Res}{Res}
\DeclareMathOperator*{\der}{der}
\DeclareMathOperator*{\diag}{diag}

\DeclareMathOperator*{\Aut}{Aut}
\DeclareMathOperator*{\id}{id}

\DeclareMathOperator*{\Hom}{Hom}

\newcommand{\mcA}{\mathcal{A}}

\newcommand{\vp}{\varphi}
\newcommand{\tvp}{\widetilde{\varphi}}

\newcommand{\tchi}{\widetilde{\chi}}
\newcommand{\teta}{\widetilde{\eta}}
\newcommand{\si}{\sigma}

\newcommand{\s}{\cong}			 
\newcommand{\pp}{\mathcal{P}} 	
\newcommand{\OO}{\mathcal{O}}

\newcommand{\ts}{\widetilde{\sigma}}

\newcommand{\CC}{\mathbb{C}}
\newcommand{\NN}{\mathbb{N}}

\newcommand{\QQ}{\mathbb{Q}}
\newcommand{\ZZ}{\mathbb{Z}}

\def\L{\mathcal L}

\def\cS{\mathcal S}

\newcommand{\tG}{\widetilde{G}}

\newcommand{\tM}{\widetilde{M}}

\title[Local Langlands Conjecture for $p$-adic $\GSpin_4,$ $\GSpin_6,$ and their inner forms]
{Local Langlands Conjecture for $p$-adic $\GSpin_4,$ $\GSpin_6,$ and their inner forms}

\author[Mahdi Asgari and Kwangho Choiy]{Mahdi Asgari and Kwangho Choiy}

\address{Mahdi Asgari\\
Department of Mathematics,
Oklahoma State University,
Stillwater, OK 74078-1058,
U.S.A.}
\email{asgari@math.okstate.edu}

\address{Kwangho Choiy\\
Department of Mathematics,
Southern Illinois University,
Carbondale, IL 62901-4408,
U.S.A.}
\email{kchoiy@siu.edu}


\begin{document}
\begin{abstract}
We establish the local Langlands conjecture for small rank general spin groups 
$\GSpin_4$ and $\GSpin_6$ as well as their inner forms.  We construct appropriate 
$L$-packets and prove that these $L$-packets satisfy the properties expected of 
them to the extent that the corresponding local factors are available. We are also able to 
determine the exact sizes of the $L$-packets in many cases. 
\end{abstract}
\maketitle 
	\section{Introduction} \label{intro}
In this article, we construct $L$-packets for the split general spin groups $\GSpin_4,$ $\GSpin_6,$ 
and their inner forms over a $p$-adic field $F$ of characteristic 0, and more importantly, 
establish their internal structures in terms of characters of component groups, as predicted 
by the Local Langlands Conjecture (LLC).  This establishes the LLC for the groups in question 
(cf. Theorems \ref{1-1 for GSpin4} and \ref{1-1 for GSpin6} and Propositions \ref{preserve4} 
and \ref{preserve6}).  The construction of the $L$-packets is essentially an exercise in 
restriction of representations, thanks to the structure, as algebraic groups, of the groups we 
consider; however, proving the properties of the $L$-packets requires some deep results 
of Hiraga-Saito as well as Aubert-Baum-Plymen-Solleveld as we explain below. 

Let $W_F$ denote the Weil group of $F.$  
In a general setting, if $G$ denotes a connected, reductive, linear, algebraic group over 
$F,$ then the Local Langlands Conjecture asserts that there is a surjective, finite-to-one 
map from the set $\Irr(G)$ of isomorphism classes of irreducible smooth complex 
representations of $G(F)$ to the set $\Phi(G)$ of $\widehat{G}$-conjugacy classes 
of $L$-parameters, i.e., admissible homomorphisms 
$\vp: W_F \times {\SL}_2(\CC) \longrightarrow {}^LG.$ 
Here $\widehat{G} = {}^LG^0$ denotes the connected component 
of the $L$-group of $G,$ i.e., the complex dual of $G$ \cite{bo79}.
Given $\vp \in \Phi(G),$ its fiber $\Pi_{\vp}(G),$ which is called an $L$-packet for $G,$ 
is expected to be controlled by a certain finite group living in the complex dual group $\widehat{G}.$
Furthermore, the map is supposed to preserve certain local factors, such as 
$\gamma$-factors, $L$-factors, and $\epsilon$-factors.

The LLC is already known for several cases: $\GL_n$ \cite{ht01, he00, scholze13}, $\SL_n$ \cite{gk82}, 
$\U_2$ and $\U_3$ \cite{rog90}, 
$F$-inner forms of $\GL_n$ and $\SL_n$ \cite{hs11, abps13},  
$\GSp_4$ \cite{gt},
$\Sp_4$ \cite{gtsp10}, 
the $F$-inner form $\GSp_{1,1}$ of $\GSp_4$ \cite{gtan12}, 
the $F$-inner form $\Sp_{1,1}$ of $\Sp_4$ \cite{choiysp11},
quasi-split orthogonal and symplectic groups \cite{art12}, 
unitary groups \cite{mok13}, and
non quasi-split inner forms of unitary groups \cite{kmsw14}.

We consider the case of $G = \GSpin_4, \GSpin_6,$ or one of their non quasi-split $F$-inner forms.  
Our approach is based on the study of the restriction of representations from a connected reductive $F$-group to 
a closed subgroup having an identical derived group as $G$ itself. This approach originates in the 
earlier work on the LLC for $\SL_n$ \cite{gk82}.  Gelbart and Knapp studied the restriction of representations 
of $\GL_n$ to $\SL_n$ and established the LLC for $\SL_n,$ assuming the LLC for $\GL_n$ which was later 
proved \cite{ht01, he00, scholze13}.  Given an $L$-parameter $\vp \in \Phi(\SL_n),$ the $L$-packet $\Pi_{\vp}(\SL_n)$ 
is proved to be in bijection with the component group $\cS_{\vp}(\widehat{\SL_n})$ of the centralizer of the image of 
$\vp$ in $\widehat{\SL_n}.$  
The multiplicity one property for the restriction from $\GL_n$ to $\SL_n$ \cite{hosh75, tad92} and the fact that all 
$L$-packets of $\GL_n$ are singletons \cite{ht01, he00, scholze13}, are indispensable to establishing the bijection.

Later on, Hiraga and Saito extended the LLC for $\SL_n$ and the result of Labesse and Langlands \cite{ll79} for 
the non-split inner form $\SL_2'$ of $\SL_2$ to the non-split inner form $\SL_n'$ of $\SL_n$ \cite{hs11}, 
except for the representations of $\SL_n'$ whose liftings to the split $\GL_n$ are not generic.  
Those cases were dealt with afterwards by Aubert, Baum, Plymen, and Solleveld \cite{abps13}.
The restriction approach is also used in the case of $\SL_n',$ 
after establishing the LLC for the non-split inner form 
$\GL_n'$ of $\GL_n$ by means of the LLC for $\GL_n$ and the local Jacquet-Langlands correspondence 
\cite{jl, dkv, rog83, ba08}.  
However, there is some subtlety in applying the restriction technique from $\GL_n'$ 
to $\SL_n'$ since, unlike in the case of split $\SL_n,$ the multiplicity one property fails in this case.  
Thus, Hiraga and Saito consider a central extension $\cS_{\vp, \scn}(\widehat{{\SL}_n})$ of the connected 
component group $\PGL_n(\CC)$ by a certain quotient group $\widehat Z_{\vp, \scn}({\SL_n})$ of the center 
of $\SL_n(\CC).$ 
They then prove that the set of irreducible representations of the finite group $\cS_{\vp, \scn}(\widehat{{\SL}_n})$ 
governs the restriction from $\GL_n'$ to $\SL_n'$ and parameterizes the $L$-packets for $\SL_n'.$  
The central extension approach turns out to also include the case of $G=\SL_n$ and the previous parameterization 
of $L$-packets of $\SL_n.$

The LLC for the groups we consider in this paper is related to the LLC for $\SL_n$ \cite{gk82} and 
its $F$-inner form $\SL_n'$ \cite{hs11,abps13}.  Write $G$ for $\GSpin_4,$ $\GSpin_6,$ 
or one of their non quasi-split $F$-inner forms. 
It follows from the structure of $G$ as an algebraic group that it is an intermediate group between a product of 
$\SL_{m_i}$ and a product of $\GL_{m_i}$ or their $F$-inner forms with suitable integers $m_i.$  We give an 
explicit description of each group structure in Section \ref{gp structure}.  We are then able to utilize the LLC for 
$\GL_n$ \cite{ht01, he00, scholze13} and $\GL_n'$ \cite{hs11}.  At the same time, using a theorem of Labesse \cite{la85}, 
we are able to define a surjective, finite-to-one map
\begin{equation} \label{L-map intro}
{\L}: \Irr(G) \longrightarrow \Phi(G),
\end{equation}
and construct $L$-packets $\Pi_{\vp}(G)$ for each $\vp \in \Phi(G).$

We next study the internal structure of each $L$-packet. Based on the work of Hiraga and Saito \cite{hs11},  
we investigate the central extension $\cS_{\vp, \scn}$ for our case and prove that 
$\cS_{\vp, \scn}$ is embedded into $\cS_{\vp, \scn}(\widehat{\SL_n}).$  
This is where the internal structures of the $L$-packets for $\SL_n$ and $\SL_n'$ are needed.
We then prove that there is a one-to-one correspondence
\begin{equation*} \label{bij intro} 
\Pi_{\vp}(G)  \, \overset{1-1}{\longleftrightarrow} \,   \Irr(\cS_{\vp, \scn}, \zeta_G) ,
\end{equation*}
where $\Irr(\cS_{\vp, \scn}, \zeta_{G})$ denotes the set of irreducible representations of $\cS_{\vp, \scn}$ 
with central character $\zeta_{G}$ corresponding 
to the group $G$ via the Kottwitz isomorphism 
\cite{kot86} (cf. Theorems \ref{1-1 for GSpin4} and \ref{1-1 for GSpin6}). 
Moreover, using Galois cohomology, we prove that the possible sizes for the $L$-packet $\Pi_{\vp}(G)$ are 
1, 2, and 4 when $p \not= 2$ and 1, 2, 4, and 8 when $p=2$ (cf. Propositions \ref{pro size 4} and \ref{pro size 6}).  
In the case of $G=\GSpin_4$ we are also able to show that only 1, 2, and 4 occur for any $p$ (see Remarks \ref{cor for I(GSpin4)} and \ref{rem 2 for I(GSpin4)}).
We do this using the classification of the group of characters stabilizing representations. 
Further, we describe the group structure of the central extension $\cS_{\vp, \scn},$ provide all the sizes of $L$-packets for $\GSpin_{4}$ 
and its non quasi-split $F$-inner forms,
and discuss the multiplicity in restriction.
In the case of $\GSpin_6,$ unlike that of $\GSpin_4,$ we do not classify the group of characters for $\GSpin_6$ 
nor do we discuss the group structure of $\cS_{\vp, \scn}.$ This is due to the fact that 
a full classification of irreducible $L$-parameters in $\Phi(\SL_4)$ is not currently available.  
Accordingly, the determination of all the sizes of $L$-packets for $\GSpin_{6}$ 
and its non quasi-split $F$-inner forms 
as well as the multiplicity in restriction are not addressed in this paper. 
These questions  will require further study of the finite group $\cS_{\vp, \scn}$ for each $L$-parameter $\vp$ for $\GSpin_6.$

Furthermore, in Sections \ref{properties for gspin4} and \ref{properties for gspin6} we verify that the Plancherel measure is preserved within $L$-packets and between inner twisting. Also, the ${\L}$-map satisfies that an irreducible representation is essentially square integrable (respectively, tempered) if and only if its $L$-parameter does not factor through any proper Levi subgroup (respectively, the image of its $L$-parameter is bounded). 

When $G$ is the split group $\GSpin_4$ or $\GSpin_6,$ we prove that the local $L$-, $\epsilon$-, and $\gamma$-factors 
are preserved via the ${\L}$-map in \eqref{L-map intro}.  Given $\tau \in \Irr(\GL_r),$ $r \ge 1,$ and $\si \in \Irr(G)$ 
which is assumed to be either $\psi$-generic or non-supercuspidal if $r > 1,$  we let $\vp_\tau$ be the $L$-parameter of $\tau$ via the LLC 
for $\GL_r$ and let $\vp_\si = \L (\si).$ 
Thanks to the structure theory detailed in Section \ref{gp structure}, we are able to use results on the generic Langlands functorial 
transfer from general spin groups to the general linear groups which are already available \cite{acs-local,ash06duke, ash14manuscripta}.  
As a result, we prove   
\begin{eqnarray} 
\nonumber 
L(s, \tau \times \si) &=& L(s, \vp_\tau \otimes \vp_\sigma), \\ 
\epsilon(s, \tau \times \si, \psi) &=& \epsilon(s, \vp_\tau \otimes \vp_\sigma, \psi), \\  
\nonumber 
\gamma(s, \tau \times \si, \psi) &=& \gamma(s, \vp_\tau \otimes \vp_\sigma, \psi). 
\end{eqnarray} 
Here, the local factors on the left hand side are those attached by Shahidi \cite[Theorem 3.5]{shahidi90annals} initially 
to generic representations and extended to all non-generic non-supercuspidal representations via the Langlands classification 
and the multiplicativity of the local factors \cite[\S 9]{shahidi90annals}, and the factors on the right hand side are Artin 
local factors associated with the given representations of the Weil-Deligne group of $F$ (cf. Sections \ref{L-factors4} and \ref{L-factors6}).

Another expected property of the $L$-packets is that they should satisfy the local character identities of the theory of (twisted) endoscopy (see \cite{gan-chan}, for example).  It is certainly desirable to study this question for the groups we consider and the $L$-packets we construct, a task we leave 
for a future work. 

We finally remark that, due to lack of the LLC for the quasi-split special unitary group $\SU_n,$ our method is currently limited to 
split groups  $\GSpin_4,$ $\GSpin_6,$ and their non quasi-split $F$-inner forms. The case of the quasi-split non-split groups $\GSpin^*_4$ 
and $\GSpin^*_6$ will be addressed in a forthcoming work.

The structure of this paper is as follows.  In Section \ref{prelim} we introduce the basic notations and review some background material. 
We also describe the algebraic group structure, $F$-points, and $L$-groups of the groups $\GSpin_4,$ $\GSpin_6,$ and 
their non quasi-split $F$-inner forms, which are the groups under consideration in this paper. 
Section \ref{general LLC} states the LLC and the conjectural structure of $L$-packets in a general setting. 
In Section \ref{results in rest} we review some well-known results on restriction.  We then prove our main results: 
the LLC for $\GSpin_4$ and its non quasi-split $F$ inner forms in Section \ref{LLC4}, and 
the LLC for $\GSpin_6$ and its non quasi-split $F$-inner forms in Section \ref{LLC6}.
Furthermore, we describe the possible sizes of $L$-packets in each case and prove the equality of local factors via the $\L$-map.

\subsection*{Acknowledgements}
This work has benefited from many helpful conversations with Wee Teck Gan and the authors would like to thank him for all his help.
The authors also thank Maarten Solleveld for his feedback and comments on this work. 
The first author was partially supported by a Collaborations Grant (\#245422) from the Simons Foundation. 
The second author was partially supported by an AMS-Simons Travel Grant. 

	\section{The Preliminaries} \label{prelim}

	\subsection{Notations and Conventions} \label{notations}
Let $p$ be a prime number.  We denote by $F$ a $p$-adic field of characteristic $0,$ i.e., a finite extension of $\QQ_{p}.$  
Let $\bar{F}$ be an algebraic closure of $F.$ Denote by $\OO_F$ the ring of integers of $F,$ and by $\pp$ the maximal ideal in $\OO_F.$ 
Let $q$ denote the cardinality of the residue field $\OO_F / \pp.$

We denote by $W_F$ the Weil group of $F$ and by $\Gamma$ the absolute Galois group $\Gal(\bar{F} / F).$ 
Let $G$ be a connected, reductive, linear, algebraic group over $F.$ 
Fixing $\Gamma$-invariant splitting data, we define the $L$-group of $G$ as a semi-direct product 
$^{L}G := \widehat{G} \rtimes \Gamma$ (see \cite[Section 2]{bo79}).

For an integer $i \in \NN$ and a connected, reductive, algebraic group $G$ over $F,$ we set 
\[ H^i(F, G) := H^i(\Gal (\bar{F} / F), G(\bar{F})), \] 
the Galois cohomology of $G.$ For any topological group $H,$ 
we denote by $\pi_0(H)$ the group $H/H^\circ$ of connected components of $H,$ where $H^\circ$ denotes the identity component of $H.$  
By $Z(H)$ we will denote the center of $H.$ 
We write $H^D$ for the group $\Hom(H , \CC^{\times})$ of all continuous characters. 
Also, $H_{\der}$ denotes the derived groups of $H.$ 
We denote by $\mathbbm{1}$ the trivial character. 
The cardinality of a finite set $S$ is denoted by $|S|.$ 
For two integers $x$ and $y,$ $x \big{|} y$ means that $y$ is divisible by $x.$
For any positive integer $n,$ we denote by $\mu_n$ the algebraic group such that $\mu_n(R) = \{ r \in R : r^n = 1 \}$ for any $F$-algebra $R.$

Given connected reductive algebraic groups $G$ and $G'$ over $F,$ we say that $G$ and $G'$ are 
\textit{$F$-inner forms} with respect to an $\bar{F}$-isomorphism $\varphi: G' \rightarrow G$ if 
$\varphi \circ \tau(\varphi)^{-1}$ is an inner automorphism ($g \mapsto xgx^{-1}$) defined over $\bar{F}$ 
for all $\tau \in \Gal (\bar{F} / F)$ (see \cite[2.4(3)]{bo79}, \cite[p.280]{kot97}).  
We often omit the references to $F$ and $\varphi$ when there is no danger of confusion.
We recall that if two $F$-inner forms $G$ and $G'$ are quasi-split over $F$, then $G$ and $G'$ are isomorphic over $F,$ \cite[Remarks 2.4(3)]{bo79}.

When $G$ and $G'$ are inner forms of each other, we have $^L{G} \s {^LG'}$ \cite[Section 2.4(3)]{bo79}. 
In particular, if $G'$ is an inner form of an $F$-split group $G$ with the action of $\Gamma$ on $\widehat{G}$ trivial, 
we write $^L{G} = \widehat{G} \s {^LG'} = \widehat{G'}.$

For positive integers $m,$ $n,$ and $d,$ we let $D$ be a central division algebra of dimension $d^2$ over $F$ (possibly $D=F,$ in which case $d=1$). 
Let $\GL_{m}(D)$ 
denote the group of all invertible elements of $m \times m$ matrices over $D.$  
Let $\SL_m(D)$ be the subgroup of elements in $\GL_m(D)$ with reduced norm $\Nrd$ equal to $1.$ 
Note that there are algebraic groups over $F,$ 
whose groups of $F$-points are respectively $\GL_m(D)$ and $\SL_m(D).$ 
By abuse of notation, we shall write $\GL_m(D)$ and $\SL_m(D)$ for their algebraic groups over $F$ as well. 
Note that any $F$-inner forms of the split general linear group $\GL_n$ and the split special linear group $\SL_n$ 
are respectively of the form ${\GL}_m(D)$ and ${\SL}_m(D),$ 
where $n=md$ (see \cite[Sections 2.2 \& 2.3]{pr94}).

	\subsection{Group Structures} \label{gp structure}
In this section we describe the structure of the split groups $\GSpin_4$ and $\GSpin_6$ over $F$ 
as well as their non quasi-split $F$-inner forms.  These are the groups we work with in this paper.  
The exact knowledge of the structure of these groups allow us, on the one hand, to use techniques 
from restriction of representations to construct $L$-packets and, on the other hand, make use of 
generic local transfer from the general spin groups to general linear groups in order to prove preservation 
of local $L$-, $\epsilon$-, and $\gamma$-factors for our $L$-packets. 

		\subsubsection{Split Groups}
We first give a description of the split groups $\GSpin_4$ and $\GSpin_6$ in terms of abstract root data.  Let 
\[ X_{2n} = \ZZ e_0 \oplus \ZZ e_1 \oplus \cdots \oplus \ZZ e_n \]
and let 
\[ X_{2n}^\vee = \ZZ e_0^* \oplus \ZZ e_1^* \oplus \cdots \oplus \ZZ e_{2n}^* \] 
be the dual $\ZZ$-module with the standard $\ZZ$-pairing between them. We let 
\[ \Delta_{2n} = \{ \alpha_1 = e_1 - e_2, \dots, \alpha_{n-1} = e_{n-1} - e_n, \alpha_n = e_{n-1} + e_n \} \]  
and 
\[ \Delta_{2n}^\vee = \{\alpha^\vee_1 = e_1^* - e_2^*, \dots \alpha^\vee_{n-1} = e_{n-1}^* - e_n^*, \alpha^\vee_n = e_{n-1}^* + e_n^* - e_0^*\} \] 
denote the simple roots and coroots, respectively, and let $R_{2n}$ and $R^\vee_{2n}$ be the roots and coroots they generate.  Then 
\[ \Psi_{2n} = (X_{2n}, R_{2n}, \Delta_{2n}, X_{2n}^\vee, R^\vee_{2n}, \Delta_{2n}^\vee) \] 
is a based root datum determining the split, connected, reductive group $\GSpin_{2n}$ over $F.$  
See \cite{acsh, ash14manuscripta} for more details.  We are mostly interested in $\Psi_4$ and $\Psi_6$ in this paper.

In fact, the derived group of $\GSpin_{2n}$ is isomorphic to $\Spin_{2n}$ and we have \cite[Proposition 2.2]{ash06duke} 
the following isomorphism of algebraic groups over $F:$  
\begin{equation} \label{GSpin-qt}
\GSpin_{2n} \s (\GL_1 \times \Spin_{2n}) / \{(1, 1), (-1, c)\}, 
\end{equation} 
where $c$ denotes the non-trivial element in the center of $\Spin_{2n}$ given by 
\[ c = \alpha^\vee_{n-1}(-1) \alpha^\vee_{n}(-1) = e_0^*(-1)^{-1} = e_0^*(-1) \] 
in our root data notation. 

When $n=2$ or $3$ we have the accidental isomorphisms 
\begin{eqnarray*} 
\Spin_4 &\s& \SL_2 \times \SL_2 \\
\Spin_6 &\s& \SL_4 
\end{eqnarray*}  
as algebraic groups over $F.$  The element $c$ will then be identified with $(-I_2, -I_2) \in \SL_2 \times \SL_2$ and 
$-I_4 \in \SL_4,$ respectively.  Therefore, we have the following isomorphisms of algebraic groups over $F:$ 

\begin{eqnarray}
\GSpin_4  &\s&  (\GL_1 \times \SL_2 \times \SL_2) / \left\{(1, I_2, I_2), (-1, -I_2, -I_2)  \right\},  \\
\GSpin_6  &\s& ({\GL}_1 \times {\Spin}_6) / \left\{(1, I_4), (-1, -I_4) \right\} .
\end{eqnarray}

For our purposes here, the following is a more convenient description of these two groups. 
\begin{pro} \label{sio for splits}
As algebraic groups over $F$ we have the following isomorphisms: 
\begin{equation} \label{convenient iso GSpin4}
{\GSpin}_4  \s \{ (g_1, g_2) \in {\GL}_2 \times {\GL}_2 : \det g_1 = \det g_2 \}
\end{equation}
\begin{equation} \label{convenient iso GSpin6}
{\GSpin}_6  \s \{ (g_1, g_2) \in {\GL}_1 \times {\GL}_4 : g_1^2 = \det g_2 \}.
\end{equation}
\end{pro}

\begin{proof}
We verify these isomorphisms by giving explicit isomorphisms between the respective root data. 
For (\ref{convenient iso GSpin4}) consider $\GL_2 \times \GL_2$ and let $T_1$ be its usual maximal split torus, 
with $T \subset T_1$ denoting the maximal split torus of  
\[ G_4 = \left\{ (g_1, g_2) \in {\GL}_2 \times {\GL}_2 : \det g_1 = \det g_2 \right\}. \]  
Using the notation $f_{ij}$ and $f^*_{ij},$ $1 \le i, j \le 2,$ for the usual $\ZZ$-basis of characters 
and cocharacters of 
$\GL_2 \times \GL_2$ with the standard $\ZZ$-pairing between them, the character lattice $X^*(T),$ a quotient of $X^*(T_1),$ and the cocharacter lattice $X_*(T),$ 
a sublattice of $X_*(T_1),$ can be given by 
\[ X^*(T)  = \frac {\ZZ \langle f_{11}, f_{12} \rangle \oplus \ZZ \langle f_{21}, f_{22} \rangle}
 {\ZZ \langle f_{11} + f_{12} - f_{21} - f_{22} \rangle} \] 
and 
\[ X_*(T)  = \langle f_{11} + f_{12} - f_{21} - f_{22} \rangle ^\perp. \] 
Here $\perp$ means orthogonal complement with respect to the $\ZZ$-pairing.  

Let $R$ and $\Delta$ denote the roots and simple roots in $G_4$ and let $R^\vee$ and $\Delta^\vee$ be coroots and simple coroots.  
An isomorphism of based root data $\Psi_4 \longrightarrow \Psi(G_4) = (X^*(T), R, \Delta, X_*(T), R^\vee, \Delta^\vee)$ amounts to isomorphisms of $\ZZ$-modules  
\[ \iota : X_4 \longrightarrow X^*(T) \quad \mbox {and } \quad \iota^\vee : X_*(T) \longrightarrow X_4^\vee \] 
such that 
\begin{eqnarray*} 
\langle \iota(x), y \rangle = \langle x, \iota^\vee(y) \rangle, &\quad x \in X_4, y \in X_*(T), \\
\iota(\Delta_4) = \Delta, &\mbox{ and } \quad \iota^\vee(\Delta^\vee) = \Delta_4^\vee 
\end{eqnarray*} 
such that the following diagram commutes:
\[ \begin{CD}
R_4 @>\vee>> R_4^\vee\\
@V\iota VV @AA{\iota^\vee}A\\
R @>\vee>> R^\vee
\end{CD} \]

Setting 
\[ \Delta = \left\{\beta_1 = f_{11} - f_{12} , \beta_2 = -f_{11} - f_{12} + 2 f_{21} \right\}  (\operatorname{mod} f_{11} + f_{12} - f_{21} - f_{22}) \] 
and 
\[ \Delta^\vee = \left\{ \beta^\vee_1 = f^*_{11} - f^*_{12}, 
\beta^\vee_2 = f^*_{21} - f^*_{22} \right\}, \]
let $S$ denote the $3 \times 3$ matrix of $\iota$ with respect to the $\ZZ$-bases $(e_0, e_1, e_2)$ of $X_4$ and 
$(f_{11}, f_{12}, f_{21}) (\operatorname{mod} f_{11} + f_{12} - f_{21} - f_{22})$ of $X^*(T),$ respectively. Similarly, let $S^\vee$ denote the matrix of 
$\iota^\vee$ with respect to the $\ZZ$-bases $(f^*_{11}+f^*_{22}, f^*_{12}+f^*_{22}, -f^*_{21}+f^*_{22})$ of $X_*(T)$ and 
$(e^*_0, e^*_1, e^*_2)$ of $X^\vee_4,$ respectively.  Assuming that $\iota(\alpha_1) = \beta_1$ and $\iota(\alpha_2) = \beta_2$ along with the 
conditions for root data isomorphisms we detailed above, plus the requirement $\det S = 1$ gives, after some computations, a unique choice for 
the $\ZZ$-isomorphisms $\iota$ and $\iota^\vee$ given by  
\begin{equation*}
S = 	\begin{bmatrix}
	0 & 0 & -1 \\
	0 & -1 & 0 \\
	-1 & 1 & 1
	\end{bmatrix}
\quad \mbox{ and } \quad 
S^\vee = {}^tS.
\end{equation*} 
(Alternatively, we could have assumed that $\det S = -1$ or that $\iota(\alpha_1) = \beta_2,$ etc., which would have led to slightly different matrices $S$ 
and $S^\vee$.) 
This shows that $G_4$ is indeed a realization, as an algebraic group over $F,$ for $\GSpin_4.$  

In fact, using $S$ and $S^\vee$ above, we can arrive at the following explicit realizations. For 
\[ t = \left( \begin{bmatrix} a& \\  &b \end{bmatrix} , \begin{bmatrix} c& \\  &d \end{bmatrix} \right) \in T,  \quad \mbox{(with $ab = cd$)} \] 
we have, 
\[ e_1(t) = c/b, \quad e_2(t) = c/a, \quad e_0(t) = 1/c. \] 
Similarly, for $\lambda \in \GL_1,$ we have 
\begin{eqnarray*} 
e^*_1(\lambda) &=& \left( \begin{bmatrix} 1& \\  &\lambda^{-1} \end{bmatrix} , \begin{bmatrix} 1& \\  &\lambda^{-1} \end{bmatrix} \right) \\
e^*_2(\lambda) &=& \left( \begin{bmatrix} 1& \\  &1 \end{bmatrix} , \begin{bmatrix} \lambda^{-1}& \\  &\lambda \end{bmatrix} \right) \\
e^*_0(\lambda) &=& \left( \begin{bmatrix} \lambda^{-1}& \\  &\lambda^{-1} \end{bmatrix} , \begin{bmatrix} \lambda^{-1}& \\  &\lambda^{-1} \end{bmatrix} \right).
\end{eqnarray*}

The proof for the isomorphism between $\GSpin_6$ and $G_6 = \left\{ (g_1, g_2) \in {\GL}_1 \times {\GL}_4 : g_1^2 = \det g_2 \right\}$ is similar. 
With notations $f_0$ and $f_1, f_2, f_3, f_4$ and their duals for the characters and cocharacters of $T_1,$ the split torus in $\GL_1 \times \GL_4,$ 
which contains the split torus $T$ of $G_4,$ we can write 
\[ X^*(T)  = \frac {\ZZ f_0 \oplus \ZZ \langle f_1, f_2, f_3, f_4 \rangle}
 {\ZZ \langle 2 f_0 - f_1 - f_2 - f_3 - f_4 \rangle} \] 
and 
\[ X_*(T)  = \langle 2 f_0 - f_1 - f_2 - f_3 - f_4 \rangle ^\perp. \] 

Again, we set  
\[ \Delta = \left\{\beta_1 = f_2 - f_3 , \beta_2 = f_1 - f_2, \beta_3 = f_3 - f_4 \right\}  (\operatorname{mod} 2 f_0 - f_1 - f_2 - f_3 - f_4) \] 
and 
\[ \Delta^\vee = \left\{ \beta^\vee_1 = f^*_2 - f^*_3, \beta^\vee_2 = f^*_1 - f^*_2, \beta^\vee_3 = f^*_3 - f^*_4 \right\}. \]
Now, we let $S$ denote the matrix of $\iota$ with respect to $(e_0, e_1, e_2, e_3)$ and 
$(f_0, f_1, f_2, f_3) (\operatorname{mod} 2 f_0 - f_1 - f_2 - f_3 - f_4)$ and let $S^\vee$ denote the matrix of $\iota^\vee$ with respect to 
$(f^*_0+2 f^*_4, f^*_1 - f^*_4, f^*_2 - f^*_4, f^*_3 - f^*_4)$ and 
$(e^*_0, e^*_1, e^*_2, e^*_3).$ Assuming $\iota(\alpha_i) = \beta_i$ for $1 \le i \le 3$ and the conditions for the root data isomorphim along with the 
requirement $\det S = 1$ we get, after similar computations, that  
\begin{equation*}
S = 	\begin{bmatrix}
	1 & -1 & -1 & -1 \\
	-1 & 1 & 1 & 0 \\
	-1 & 1 & 0 & 1 \\
	-1 & 0 & 1 & 1 \\
	\end{bmatrix}
\quad \mbox{ and } \quad 
S^\vee = {}^tS = S.
\end{equation*}  
This shows that $G_6$ is indeed a realization, as an algebraic group over $F,$ for $\GSpin_6.$  

Again, with the above $S$ we have the following explicit realizations. For 
\[ t = \left( z, \diag(a, b, c, d) \right) \in T \quad \mbox{(with $z^2 = abcd$)} \] 
we have 
\begin{eqnarray*} 
e_1(t) &=& a b z^{-1} \\
e_2(t) &=& a c z^{-1} \\
e_3(t) &=& b c z^{-1} \\
e_0(t) &=& d z^{-1} \\
\end{eqnarray*} 
Similarly, for $\lambda \in \GL_1,$ we have 
\begin{eqnarray*} 
e^*_1(\lambda) &=& \left( \lambda^{-1},  \diag(1,1,\lambda^{-1}, \lambda^{-1}) \right) \\
e^*_2(\lambda) &=& \left( \lambda^{-1},  \diag(1,\lambda^{-1}, 1, \lambda^{-1}) \right) \\
e^*_3(\lambda) &=& \left( \lambda^{-1},  \diag(\lambda^{-1}, 1, 1, \lambda^{-1}) \right) \\
e^*_0(\lambda) &=& \left( \lambda^{-1},  \diag(\lambda^{-1}, \lambda^{-1}, \lambda^{-1}, \lambda^{-1}) \right). \\
\end{eqnarray*}
\end{proof}
 
The proposition implies that we have the following inclusions
\begin{eqnarray}
\label{derived of 4}
{\SL}_2 \times {\SL}_2 \subset&  {\GSpin}_4 \subset& {\GL}_2 \times {\GL}_2  \\
\label{derived of 6}
{\SL}_4  \subset&  {\GSpin}_6 \subset& {\GL}_1 \times {\GL}_4.
\end{eqnarray}
Since 
\begin{eqnarray*}
{\SL}_2 \times {\SL}_2 = ({\GSpin}_4)_{\der} = ({\GL}_2 \times {\GL}_2)_{\der} \\
{\SL}_4  = ({\GSpin}_6)_{\der} = ({\GL}_1 \times {\GL}_4)_{\der}
\end{eqnarray*}
the two groups ${\GSpin}_4$ and ${\GSpin}_6$ fit in the setting (cf. \eqref{cond on G})  
\begin{equation} \label{cond on G-early}
G_{\der} = \tG_{\der} \subseteq G \subseteq \tG,
\end{equation}
with 
\[ G=\GSpin_4, \quad \tG = \GL_2 \times \GL_2 \] 
and 
\[ G=\GSpin_6, \quad \tG = \GL_1 \times \GL_v, \] 
respectively.

Moreover, using the surjective map ${\GL}_2 \times {\GL}_2 \longrightarrow {\GL}_1$ defined by by $(g_1, g_2) \mapsto (\det g_1)(\det g_2)^{-1},$ 
the isomorphism \eqref{convenient iso GSpin4} gives an exact sequence of algebraic groups
\begin{equation} \label{convenient exact sequence GSpin4}
1 \longrightarrow {\GSpin}_4 \longrightarrow {\GL}_2 \times {\GL}_2 \longrightarrow {\GL}_1 \longrightarrow 1.
\end{equation}
Likewise, using the surjective map ${\GL}_1 \times {\GL}_4 \longrightarrow {\GL}_1$ defined by $(g_1, g_2) \mapsto g_1^{-2}(\det g_2),$ 
the isomorphism \eqref{convenient iso GSpin6} yields the exact sequence 
\begin{equation} \label{convenient exact sequence GSpin6}
1 \longrightarrow {\GSpin}_6 \longrightarrow {\GL}_1 \times {\GL}_4 \longrightarrow {\GL}_1 \longrightarrow 1.
\end{equation}

		\subsubsection{Inner Forms}
Using the Satake classification \cite[p.119]{sa71} for admissible diagrams of the $F$-inner forms, we only have two (up to isomorphism) non quasi-split 
$F$-inner forms of ${\GSpin}_4,$ denoted by ${\GSpin}^{2,1}_4$ and ${\GSpin}^{1,1}_4,$ whose diagrams are respectively
\[
\xy
\POS (20,5) *\cir<2pt>{} ="b",
\POS (20,-5) *{\bullet} ="c",
\POS (35,0) *{\text{and}},

\POS (50,5) *{\bullet} ="g",
\POS (50,-5) *{\bullet} ="h",
\endxy
\]
The left diagram gives
\begin{equation} \label{derived of 4 inner 1}
{\SL}_2 \times {\SL}_1(D) \subset {\GSpin}^{2,1}_4 \subset {\GL}_2 \times {\GL}_1(D)
\end{equation}
and the right one gives
\begin{equation} \label{derived of 4 inner 2}
{\SL}_1(D) \times {\SL}_1(D) \subset {\GSpin}^{1,1}_4 \subset {\GL}_1(D) \times {\GL}_1(D).
\end{equation}
Here, $D$ denotes the quaternion division algebra over $F.$  (Recall from Section \ref{notations} that we are writing 
$\GL_m(D)$ and $\SL_m(D)$ for both algebraic groups over $F$ and their $F$-points, by abuse of notation.) 

Similarly, we only have two (up to isomorphism) non quasi-split $F$-inner forms of 
${\GSpin}_6,$ denoted by ${\GSpin}_6^{2,0}$ and ${\GSpin}_6^{1,0},$ whose diagrams are respectively 
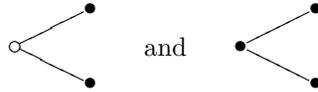
\[
\xy
\POS (10,0) *\cir<2pt>{} ="a",
\POS (20,5) *{\bullet} ="b",
\POS (20,-5) *{\bullet} ="c",

\POS (30,0) *{\text{and}},

\POS (40,0) *{\bullet} ="f",
\POS (50,5) *{\bullet} ="g",
\POS (50,-5) *{\bullet} ="h",

\POS "a" \ar@{-}^<<{}_<<<{} "b",
\POS "a" \ar@{-}^<<<{}_<<{} "c",

\POS "f" \ar@{-}^<<{}_<<<{} "g",
\POS "f" \ar@{-}^<<<{}_<<{} "h",
\endxy
\]
The left diagram gives
\begin{equation} \label{derived of 6 inner 1}
{\SL}_2(D) \subset {\GSpin}_6^{2,0} \subset {\GL}_1 \times {\GL}_2(D)
\end{equation}
and the right one gives
\begin{equation} \label{derived of 6 inner 2}
{\SL}_1(D_4) \subset {\GSpin}_6^{1,0} \subset {\GL}_1 \times {\GL}_1(D_4).
\end{equation}
Here, $D_4$ is a division algebra of dimension $16$ over $F.$
We note that the two division algebras $D_4$ and its opposite $D_4^{op}$ of dimension 16 ( with invariants $1/4,$ $-1/4$ in $\QQ/\ZZ$) 
have canonically isomorphic multiplicative groups $D_4^{\times}$ and $(D_4^{op})^{\times}.$

Similar to the split forms $\GSpin_4$ and $\GSpin_6,$ the $F$-inner forms of $\GSpin_4$ and $\GSpin_6$ as well as the $F$-inner forms of 
${\SL}_2 \times {\SL}_2,$ ${\GL}_2 \times {\GL}_2,$ ${\SL}_4,$ and ${\GL}_4$ all satisfy 
the setting \eqref{cond on G-early}.

The arguments giving \eqref{convenient exact sequence GSpin4} and \eqref{convenient exact sequence GSpin6} apply again, with 
$\det$ replaced by the reduced norm $\Nrd,$ to show that for $F$-inner forms 
${\GSpin}^{2,1}_4,$ ${\GSpin}^{1,1}_4,$ ${\GSpin}_6^{2,0},$ and ${\GSpin}_6^{1,0},$ 
we have 
\begin{eqnarray} 
\label{convenient exact sequence 4 inner 1}
1 \longrightarrow & {\GSpin}^{2,1}_4 \longrightarrow & {\GL}_2 \times {\GL}_1(D) \longrightarrow {\GL}_1 \longrightarrow 1, 
\\
\label{convenient exact sequence 4 inner 2}
1 \longrightarrow & {\GSpin}^{1,1}_4 \longrightarrow & {\GL}_1(D) \times {\GL}_1(D) \longrightarrow {\GL}_1 \longrightarrow 1,
\\
\label{convenient exact sequence 6 inner 1}
1 \longrightarrow & {\GSpin}_6^{2,0} \longrightarrow & {\GL}_1 \times {\GL}_2(D) \longrightarrow {\GL}_1 \longrightarrow 1,
\\
\label{convenient exact sequence 6 inner 2}
1 \longrightarrow & {\GSpin}_6^{1,0}\longrightarrow & {\GL}_1 \times {\GL}_1(D_4) \longrightarrow {\GL}_1 \longrightarrow 1.
\end{eqnarray}
Using \eqref{convenient exact sequence 4 inner 1} -- \eqref{convenient exact sequence 6 inner 2}, we have the following isomorphisms of algebraic groups over $F,$ which are analogues of Proposition \ref{sio for splits}.
\begin{eqnarray*} 
&{\GSpin}^{2,1}_4 \s & \{ (g_1,g_2) \in {\GL}_2 \times {\GL}_1(D) : \det g_1 = \Nrd g_2 \}, 
\\
&{\GSpin}^{1,1}_4 \s & \{ (g_1,g_2) \in {\GL}_1(D) \times {\GL}_1(D) : \Nrd g_1 = \Nrd g_2 \},
\\
&{\GSpin}_6^{2,0} \s & \{ (g_1,g_2) \in {\GL}_1 \times {\GL}_2(D): g^2_1 = \Nrd g_2 \},
\\
&{\GSpin}_6^{1,0} \s & \{ (g_1,g_2) \in {\GL}_1 \times {\GL}_1(D_4): g^2_1 = \Nrd g_2 \}.
\end{eqnarray*}

Finally, recall \cite[Proposition 2.3]{ash06duke} that we have
\begin{equation} \label{centers}
\pi_0(Z({\GSpin}_4)) \s \pi_0(Z({\GSpin}_6)) \s \ZZ/2\ZZ.
\end{equation}
This also holds for their $F$-inner forms since $F$-inner forms have the same center.

	\subsection{The $F$-points}  \label{F points} 
We now describe the $F$-rational points of the split groups $\GSpin_4$ and $\GSpin_6$ as well as their non quasi-split inner forms.  
We start with the following. 

\begin{lm} \label{GC for GSpin}
Let $G$ be an $F$-inner form of ${\GSpin}_{2n}$ with $n \ge 1$ an integer. Then 
\[
H^1(F, G) = 1.
\]
\end{lm}

\begin{proof}
We have \cite[Propositions 2.2 and 2.10]{ash06duke} 
\[
\widehat{{\GSpin}_{2n}}={\GSO}_{2n}(\CC)
\]
and 
\[
Z(\widehat{{\GSpin}_{2n}})^{\Gamma} = Z({\GSO}_{2n}(\CC)) \s \CC^{\times}.
\]
Here, $\widehat{\ }$ denotes the connected component of the $L$-group.  Applying the Kottwitz isomorphism \cite[Theorem 1.2]{kot86}, we can conclude 
\[
H^1(F, {\GSpin}_{2n}) \s \pi_0(Z(\widehat{{\GSpin}_{2n}})^{\Gamma}) ^D = \pi_0(\CC^{\times}) ^D= 1.
\]
Since the Kottwitz isomorphism is valid for any connected reductive algebraic group over a $p$-adic field of characteristic $0,$ and since $F$-inner forms of algebraic groups share the same $L$-groups (cf. Section \ref{notations}), the proof is complete.
\end{proof}

Also, recall that $H^1(F, \GL_n)=1$ and $H^1(F, \Spin_n)=1$ for $n \ge 1$ (see \cite[Lemma 2.8 and Theorem 6.4]{pr94}, for example).   
It follows from \eqref{GSpin-qt} that the group $\GSpin_4(F)$ can be described via  
\[
1 \longrightarrow \{ \pm1\} \longrightarrow F^{\times} \times {\Spin}_4(F) \longrightarrow {\GSpin}_4(F) \longrightarrow H^1(F, \{ \pm1\}) \longrightarrow 1
\] 
or 
\[
1 \longrightarrow (F^{\times} \times {\Spin}_4(F))/\{ \pm1\} \longrightarrow  {\GSpin}_4(F) \longrightarrow F^{\times}/(F^{\times})^2 \longrightarrow 1.
\]
Likewise, the group $\GSpin_6(F)$ can be described via 
\[
1 \longrightarrow \{ \pm1 \} \longrightarrow F^{\times} \times {\Spin}_6(F) \longrightarrow {\GSpin}_6(F) \longrightarrow H^1(F, \{ \pm1\}) \longrightarrow 1
\] 
or
\[
1 \longrightarrow (F^{\times} \times {\Spin}_6(F))/\{ \pm1\} \longrightarrow  {\GSpin}_6(F) \longrightarrow F^{\times}/(F^{\times})^2 \longrightarrow 1.
\]

Using Lemma \ref{GC for GSpin}, we apply Galois cohomology to \eqref{convenient exact sequence GSpin4} to obtain 
\begin{equation} \label{GSpin4-F}
1 \longrightarrow {\GSpin}_4(F) \longrightarrow {\GL}_2(F) \times {\GL}_2(F) \longrightarrow F^{\times} \longrightarrow 1.
\end{equation}
Thus, we have 
\begin{equation} \label{convenient description of GSpin4(F)}
{\GSpin}_4(F) \s \{(g_1, g_2) \in {\GL}_2(F) \times {\GL}_2(F) : \det g_1 = \det g_2 \}.
\end{equation}
Likewise, using Lemma \ref{GC for GSpin} and \eqref{convenient exact sequence GSpin6} we obtain 
\begin{equation} 
1 \longrightarrow {\GSpin}_6(F) \longrightarrow {\GL}_1(F) \times {\GL}_4(F) \longrightarrow F^{\times} \longrightarrow 1.
\end{equation}
Thus, we have 
\begin{equation} \label{convenient description of GSpin6(F)}
{\GSpin}_6(F) \s \{(g_1, g_2) \in {\GL}_1(F) \times {\GL}_2(F) : (g_1)^{-2} = \det g_2 \}.
\end{equation}
Similarly, using \eqref{convenient exact sequence 4 inner 1} -- \eqref{convenient exact sequence 6 inner 2} we get 
\begin{eqnarray*} 
1 \longrightarrow & {\GSpin}^{2,1}_4(F) \longrightarrow & {\GL}_2(F) \times {\GL}_1(D) \longrightarrow  F^{\times} \longrightarrow 1,
\\
1 \longrightarrow & {\GSpin}^{1,1}_4(F) \longrightarrow & {\GL}_1(D) \times {\GL}_1(D) \longrightarrow  F^{\times} \longrightarrow 1,
\\
1 \longrightarrow & {\GSpin}_6^{2,0}(F) \longrightarrow & F^{\times} \times {\GL}_2(D)  \longrightarrow  F^{\times} \longrightarrow 1,
\\
1 \longrightarrow & {\GSpin}_6^{1,0}(F) \longrightarrow & F^{\times} \times {\GL}_1(D_4) \longrightarrow  F^{\times} \longrightarrow 1.
\end{eqnarray*}
We thus have
\begin{eqnarray*} 
&{\GSpin}^{2,1}_4(F) \s & \{ (g_1,g_2) \in {\GL}_2(F) \times {\GL}_1(D) : \det g_1 = \Nrd g_2 \}, 
\\
&{\GSpin}^{1,1}_4(F) \s & \{ (g_1,g_2) \in {\GL}_1(D) \times {\GL}_1(D) : \Nrd g_1 = \Nrd g_2 \},
\\
&{\GSpin}_6^{2,0}(F) \s & \{ (g_1,g_2) \in F^{\times} \times {\GL}_2(D): g^2_1 = \Nrd g_2 \},
\\
&{\GSpin}_6^{1,0}(F) \s & \{ (g_1,g_2) \in F^{\times} \times {\GL}_1(D_4): g^2_1 = \Nrd g_2 \}.
\end{eqnarray*}

	\subsection{$L$-groups} \label{L-groups}
We recall the following descriptions of the $L$-groups of (the split groups) $\GSpin_4$ and $\GSpin_6$ 
from \cite[Proposition 2.2]{ash06duke} and \cite[Sections 3 and 4]{gt}:
\begin{equation} \label{L-gp for GSpin4}
{^L{\GSpin}_4} = \widehat{{\GSpin}_4}  = {\GSO}_{4}(\CC) \s  ({\GL}_2(\CC) \times {\GL}_2(\CC))  / \{ (z^{-1}, z) : z \in \CC^{\times}   \},  
\end{equation}
\begin{equation} \label{L-gp for GSpin6}
{^L{\GSpin}_6}  = \widehat{{\GSpin}_6}  = {\GSO}_{6}(\CC)  \s  ({\GL}_1(\CC) \times {\GL}_4(\CC))  / \{ (z^{-2}, z) : z \in \CC^{\times}   \}.
\end{equation}
This immediately gives 
\begin{equation} \label{ex seq L-gp for GSpin4}
1 \longrightarrow \CC^{\times} \longrightarrow {\GL}_2(\CC) \times {\GL}_2(\CC) \overset{pr_4}{\longrightarrow} \widehat{{\GSpin}_4} \longrightarrow 1 
\end{equation} 
and 
\begin{equation} \label{ex seq L-gp for GSpin6}
1 \longrightarrow \CC^{\times} \longrightarrow {\GL}_1(\CC) \times {\GL}_4(\CC) \overset{pr_6}{\longrightarrow} \widehat{{\GSpin}_6} \longrightarrow 1. 
\end{equation}
Here, $\CC^{\times}$ is considered as $\widehat{\GL}_1$ in \eqref{convenient exact sequence GSpin4} and 
\eqref{convenient exact sequence GSpin6}. 
Further, the injection $\CC^{\times} \hookrightarrow {\GL}_2(\CC) \times {\GL}_2(\CC)$ is given by $g \longmapsto (g^{-1}I_2, gI_2)$ 
and the injection $\CC^{\times} \hookrightarrow {\GL}_1(\CC) \times {\GL}_4(\CC)$ is given by $g \longmapsto (g^{-2}, gI_4).$

The inclusions in \eqref{derived of 4} and \eqref{derived of 6} provide the following surjective maps: 
\begin{equation*}
{\GL}_2(\CC) \times {\GL}_2(\CC) \overset{pr_4}{\twoheadrightarrow} \widehat{{\GSpin}_4} 
\twoheadrightarrow {\PGL}_2(\CC) \times {\PGL}_2(\CC) = \widehat{{\SL}_2 \times {\SL}_2}
\end{equation*} 
\begin{equation*}
{\GL}_1(\CC) \times {\GL}_4(\CC) \overset{pr_6}{\twoheadrightarrow} \widehat{{\GSpin}_6} 
\twoheadrightarrow {\PGL}_4(\CC) = \widehat{{\SL}_4}.
\end{equation*} 
Since $F$-inner forms of algebraic groups share the same $L$-groups (cf. Section \ref{notations}), we have 
\begin{align*}
{^L{\GSpin}_4} & = {^L{\GSpin}^{2,1}_4} = {^L{\GSpin}^{1,1}_4}
\\
{^L{\GSpin}_6} & = {^L{\GSpin}_6^{2,0}} = {^L{\GSpin}_6^{1,0}}.
\end{align*}

\section{Local Langlands Conjecture in a General Setting} \label{general LLC}
In this section we quickly review some generalities about the Local Langlands Conjecture (LLC). 
Let $G$ be a connected, reductive, algebraic group over $F.$ 
Write $G(F)$ for the group of $F$-points of $G.$
Let $\Irr(G)$ denote the set of equivalence classes of irreducible, smooth, complex representations of $G(F).$
By abuse of notation, we identify an equivalence class with its representatives. 
We write $\Pi_{\disc}(G)$ and $\Pi_{\temp}(G)$ for the subsets of $\Irr(G)$ consisting, respectively, of 
discrete series and tempered representations.  Moreover, we write $\Pi_{\scup}(G)$ for the subset in $\Irr(G)$ consisting of supercuspidal ones. 
Furthermore, we write 
\[ \
\Pi_{\scup, \unit}(G) = \Pi_{\scup}(G) \cap \Pi_{\disc}(G). 
\]
Note that we have
\begin{equation} \label{rep-include}
\Pi_{\scup, \unit}(G) \subset \Pi_{\disc}(G) \subset \Pi_{\temp}(G) \subset \Irr(G).
\end{equation}
Let $\Phi(G)$ denote the set of $\widehat{G}$-conjugacy classes of $L$-parameters, i.e., admissible homomorphisms 
\[
\vp: W_F \times {\SL}_2(\CC) \longrightarrow {^L}G,
\]
(see \cite[Section 8.2]{bo79}). 
We denote the centralizer of the image of $\vp$ in $\widehat{G}$ by $C_{\vp}.$  The center of $^{L}G$ is the $\Gamma$-invariant group $Z(\widehat{G})^{\Gamma}.$ Note that 
$C_{\vp} \supset Z(\widehat{G})^{\Gamma}.$ 
We say that $\vp$ is elliptic if the quotient group $C_{\vp} / Z(\widehat{G})^{\Gamma}$ is finite, and $\vp$ is tempered if $\vp(W_F)$ is bounded. 
We denote by $\Phi_{\el}(G)$ and $\Phi_{\temp}(G)$ the subset of $\Phi(G)$ which consist, respectively, 
of elliptic and tempered $L$-parameters of $G.$ 
We set 
\[ 
\Phi_{\disc}(G) = \Phi_{\el}(G) \cap \Phi_{\temp}(G). 
\] 
Moreover, we write $\Phi_{\simi}(G)$ for the subset in $\Phi(G)$ consisting of irreducible ones. 
Furthermore, let 
\[ 
\Phi_{\simi, \disc}(G) = \Phi_{\simi}(G) \cap \Phi_{\disc}(G). 
\] 
We then have, in parallel to \eqref{rep-include}, 
\begin{equation} \label{par-include}
\Phi_{\simi, \disc}(G) \subset \Phi_{\disc}(G) \subset \Phi_{\temp}(G) \subset \Phi(G).
\end{equation} 

For any connected reductive group $G$ over $F,$ 
the Local Langlands Conjecture predicts that there is a surjective finite-to-one map
\[
\Irr(G) \longrightarrow \Phi(G).
\]
This map is supposed to satisfy a number of natural properties.  For instance, it preserves 
certain $\gamma$-facotrs, $L$-factors, and $\epsilon$-factors, which one can attach to both sides.
Moreover, considering the fibers of the map, one can partition $\Irr(G)$ into a disjoint union of finite subsets, called the $L$-packets. 
Each packet is conjectured to be characterized by component groups in the $L$-group, which, for groups we are considering in 
this paper, are discussed in Sections \ref{para for GSpin4} and \ref{para for GSpin6}. 
It is also expected that $\Phi_{\disc}(G)$ and $\Phi_{\temp}(G)$ parameterize $\Pi_{\disc}(G)$ and $\Pi_{\temp}(G),$ respectively.

Denote by $\widehat{G}_{\scn}$ the simply connected cover of the derived group $\widehat{G}_{\der}$ of $\widehat{G},$ and by $\widehat{G}_{\ad}$ the quotient group $\widehat{G}/Z(\widehat{G}).$
We consider 
\[
S_{\vp} := C_{\vp} / Z(\widehat{G})^{\Gamma} \subset \widehat{G}_{\ad}.
\]
Write $S_{\vp, \scn}$ for the full pre-image of $S_{\vp}$ in $\widehat{G}_{\scn}.$ 
We then have an exact sequence
\begin{equation} \label{exact isogeny}
1 \longrightarrow Z(\widehat G_{\scn}) \longrightarrow S_{\vp, \scn} \longrightarrow S_{\vp} \longrightarrow 1.
\end{equation}
We let 
\begin{align*}
\cS_{\vp} &:= \pi_0(S_{\vp}) \\
\cS_{\vp, \scn} & := \pi_0(S_{\vp, \scn}) \\
\widehat Z_{\vp, \scn} & := Z(\widehat G_{\scn}) / (Z(\widehat G_{\scn}) \cap S_{\vp, \scn}^{\circ}). 
\end{align*}
We then have (see \cite[(9.2.2)]{art12}) a central extension 
\begin{equation} \label{central ext}
1 \longrightarrow \widehat Z_{\vp, \scn}  \longrightarrow \cS_{\vp, \scn} \longrightarrow \cS_{\vp} \longrightarrow 1. 
\end{equation}

Next, let $G'$ be an $F$-inner form of $G.$
Fix a character $\zeta_{G'}$ of 
$Z(\widehat G_{\scn})$ whose restriction to $Z(\widehat G_{\scn})^{\Gamma}$ corresponds to 
the class of the $F$-inner form $G'$ via the Kottwitz isomorphism \cite[Theorem 1.2]{kot86}.  
We denote by $\Irr(\cS_{\vp, \scn}, \zeta_{G'})$ the set of irreducible representations of $\cS_{\vp, \scn}$ 
with central character $\zeta_{G'}$ on  $Z(\widehat G_{\scn}).$ 
It is expected that, given an $L$-parameter $\vp$ for $G',$ 
there is a bijection between the $L$-packet $\Pi_{\vp}(G')$ associated to $\vp$ and the set 
$\Irr(\cS_{\vp, \scn}, \zeta_{G'})$ \cite[Section 3]{art06}. 
We note that when $G'=G$ the character $\zeta_{G'}$ is the trivial character $\mathbbm{1}$ so that 
\[
\Irr(\cS_{\vp, \scn}, \mathbbm{1}) = \Irr(\cS_{\vp}).
\]
In particular, if $\vp$ is elliptic,  then we have $S_{\vp}=\cS_{\vp}$ and 
$\widehat Z_{\vp, \scn} = Z(\widehat G_{\scn})$
since $C_{\phi} / Z(\widehat{G})^{\Gamma}$ is finite and $Z(\widehat{G})^{\Gamma}$ 
contains $S_{\phi}^{\circ}$ \cite[Lemma 10.3.1]{kot84}.   
Thus the exact sequence \eqref{central ext} turns out to be the same as \eqref{exact isogeny}.

\section{Review of Results on Restriction} \label{results in rest}
In this section, we review several results about restriction.  We refer to \cite{gk82, la85, tad92, hs11} for details.
\subsection{Results of  Gelbart-Knapp, Tadi\'c, and Hiraga-Saito} \label{results gelbart-knapp-tadic-hiraga-saito}
For this section, we let $G$ and $\tG$ be connected, reductive, algebraic groups over $F$ satisfying the property that
\begin{equation} \label{cond on G}
G_{\der} = \tG_{\der} \subseteq G \subseteq \tG,
\end{equation}
where the subscript ``$\der$" stands for the derived group. 
Given $\sigma \in \Irr(G),$ by \cite[Lemma 2.3]{gk82} and \cite[Proposition 2.2]{tad92}, there exists $\ts \in \Irr(\tG)$ such that 
\begin{equation} \label{cond on sigma}
\sigma \hookrightarrow {\Res}_{G}^{\tG}(\ts). 
\end{equation}
In other words, $\sigma$ is an irreducible constituent in the restriction ${\Res}_{G}^{\tG}(\ts)$ of $\ts$ from $\tG(F)$ to $G(F).$
It turns out, \cite[Proposition 2.4 \& Corollary 2.5]{tad92} and \cite[Lemma 2.1]{gk82}, that $\Pi_{\sigma}(G)$ 
is finite and independent of the choice of the lifting $\ts \in \Irr(\tG).$  We write $\Pi_{\sigma}(G) = \Pi_{\ts}(G)$ 
for the set of equivalence classes of all irreducible constituents of ${\Res}_{G}^{\tG}(\ts).$ 
It is clear that for any irreducible constituents $\si_1$ and  $\si_2$ in ${\Res}_{G}^{\tG}(\ts),$ 
we have $\Pi_{\si_1}(G) = \Pi_{\si_2}(G).$

\begin{rem} \label{rem in rest from to}
A member (equivalently all members) of $\Pi_{\ts}(G)$ is supercuspidal, essentially square-integrable, or essentially tempered if and only if 
$\ts$ is (see \cite[Proposition 2.7]{tad92}). 
\end{rem}

We recall that the stabilizer of $\sigma$ in $\tG$ is defined as 
\[
\tG_{\sigma}:= \left\{ \tilde{g} \in \tG(F) : {^{\tilde{g}}{\sigma}} \s \sigma \right\}.
\]
The quotient of $\tG(F)/\tG_{\sigma}$ acts by conjugation on the set $\Pi_{\ts}(G)$ simply transitively (see \cite[Lemma 2.1(c)]{gk82}) 
and there is a bijection between $\tG(F)/\tG_{\sigma}$ and $\Pi_{\ts}(G).$

We also recall the following useful result. 
\begin{pro}(\cite[Lemma 2.4]{gk82}, \cite[Corollary 2.5]{tad92}, and \cite[Lemma 2.2]{hs11}) \label{pro for lifting}
Let $\ts_1$, $\ts_2 \in \Irr(\tG).$  The following statements are equivalent:
\begin{enumerate}[(1)]
  \item  There exists a character $\chi \in (\tG(F) / G(F))^D$ such that $\ts_1 \s \ts_2 \otimes \chi;$
  \item  $\Pi_{\ts_1}(G) \cap \Pi_{\ts_2}(G) \neq \emptyset;$
  \item  $\Pi_{\ts_1}(G) = \Pi_{\ts_2}(G).$
\end{enumerate}
\end{pro}
Since ${\Res}_{G}^{\tG}(\ts)$ is completely reducible by \cite[Lemma 2.1]{gk82} and \cite[Lemma 2.1]{tad92}, we have the decomposition
\begin{equation} \label{decomp of Res}
{\Res}_{G}^{\tG}(\ts) = m \bigoplus _{\tau \in \Pi_{\ts}(G)}  \tau
\end{equation}
(see \cite[Chapter 2]{hs11}),
where the positive integer $m$ denotes the common multiplicity over $\tau \in \Pi_{\sigma}(G)$ \cite[Lemma 2.1(b)]{gk82}. 
Given $\ts \in \Irr(\tG),$ we define 
\begin{equation} \label{def of I sigma}
I(\ts) := \left\{ \chi \in (\tG(F)/G(F))^D : \ts \otimes \chi \s \ts  \right\}.
\end{equation}
Considering the dimension of the $\CC$-vector space $\End_G({\Res}_{G}^{\tG}(\ts)),$ 
we have (see \cite[Proposition 3.1]{choiymulti}) the equality 
\begin{equation} \label{multi and characters}
m^2 \cdot |\Pi_{\sigma}(G)| = I(\ts).
\end{equation}
Following \cite[Chapter 2]{hs11}, since $\ts \s \ts \otimes \chi$ for $\chi \in I(\ts),$ 
we have a non-zero endomorphism $I_{\chi} \in {\Aut}_{\CC}(V_{\ts})$ such that
\[
I_{\chi} \circ (\ts \otimes \chi)  = \ts \circ I_{\chi}.
\]
For each $z \in \CC^\times,$ we denote by $z \cdot \id_{V_{\ts}}$ the scalar endomorphism $\tilde v \mapsto z \cdot \tilde v$ for $v \in V_{\ts}.$ 
We identify $\CC^\times$ with the subgroup of ${\Aut}_{\CC}(V_{\ts})$ consisting of $z \cdot \id_{V_{\ts}}.$
Define $\mcA(\ts)$ as the subgroup of ${\Aut}_{\CC}(V_{\ts})$ generated by $\{I_{\chi} : \chi \in I(\ts) \}$ 
and $\CC^\times.$ Then the map $I_{\chi} \mapsto \chi$ induces the exact sequence
\begin{equation} \label{exact for CC1}
1 \longrightarrow \CC^\times \longrightarrow \mcA(\ts) \longrightarrow I(\ts) \longrightarrow 1.
\end{equation}
We denote by $\Irr(\mcA(\ts), \id)$ the set of isomorphism classes of irreducible smooth representations 
of the  group $\mcA(\ts)$ such that $z\cdot \id_{V_{\ts}} \in \CC^\times$ acts as the scalar $z.$ 
By \cite[Corollary 2.10]{hs11}, we then have a decomposition
\begin{equation}  \label{useful decomp}
V_{\ts} ~ ~ \s \bigoplus_{\xi \in \Irr(\mcA(\ts), \id)} \xi \boxtimes  \si_{\xi}
\end{equation}
as representations of the direct product $\mcA(\ts) \times G(F).$
 It follows that there is a one-to-one correspondence
\begin{equation}  \label{bij 333}
\Irr(\mcA(\ts), \id) \s \Pi_{\ts}(G),
\end{equation}
sending $\xi \mapsto  \si_{\xi}.$ 
We denote by $\xi_{\si}$ the preimage of $\sigma$ via the correspondence \eqref{bij 333}.

\subsection{A Theorem of Labesse} \label{thm by Labesse}
We recall a theorem of Labesse in \cite{la85} which verifies the existence of a lifting of a given $L$-parameter in the following setting. 
Let $G$ and $\tG$ be connected, reductive, algebraic groups over $F$ with an exact sequence of connected components of $L$-groups
\[
1 \longrightarrow \widehat{S} \longrightarrow \widehat{\tG} \overset{pr}{\longrightarrow}\widehat{G} \longrightarrow 1,
\]
where $\widehat{S}$ is a central torus in $\widehat{\tG},$ 
and the surjective homomorphism $pr$ is compatible with 
$\Gamma$-actions on $\widehat{\tG}$ and $\widehat{G}.$
Then, Labesse proves in \cite[Th\'{e}or\`{e}m 8.1]{la85} that for any $\vp \in \Pi(G),$ there exists $\tvp \in \Pi(\tG)$ such that
\[
\vp = \tvp \circ pr.
\]
We note that the analogous result has been proved in \cite{weil74, he80, gtsp10} for the case $G=SL_n$ and $\tG=GL_n.$ 

\subsection{$L$-packets for Inner Forms of $SL_n$} \label{for sl(m,D)}
In this section we recall some results about the LLC for $\SL_m(D)$ in \cite[Chapter 12]{hs11} and \cite[Section 3]{abps13}.
Throughout this section, let $G(F)=\SL_m(D)$ and $G^*(F)=\SL_n(F),$ where $D$ be a central division algebra of dimension $d^2$ over $F$ with $n=md$ (possibly $D=F,$ in which case $d=1$). 

Since $\Gamma$ acts on $\widehat{G}$ trivially, we shall use $\widehat{G} = {}^LG^0$ instead of $^L G = \widehat{G} \times \Gamma.$  
Note that $\widehat{G}=\widehat{G^*} = \PGL(\CC).$
Let 
\[ 
\varphi : W_F \times SL_2(\CC) \rightarrow \widehat{G}
\] 
be an $L$--parameter. Note that
\[
{Z}(\widehat{G}_{\scn}) = \mathbf{\mu_n}(\CC) \quad \text{and} \quad Z(\widehat{G})^{\Gamma} = 1.
\]
With notation as in Section \ref{general LLC}, we have the exact sequence
\begin{equation} \label{exact seq for S-gps}
1 \longrightarrow \widehat Z_{\vp, \scn}  \longrightarrow \cS_{\vp, \scn} \longrightarrow \cS_{\vp} \longrightarrow 1.
\end{equation}
In the case at hand, $\widehat Z_{\vp, \scn} = \mathbf{\mu_n}(\CC) / (\mathbf{\mu_n}(\CC)  \cap S_{\vp, \scn}^{\circ}).$
In particular, when $\varphi$ is elliptic, 
since both $S_{\vp, \scn}$ in $\SL_n(\CC)$ and $S_{\vp}$ in $\PGL_n(\CC)$ are finite, the exact sequence \eqref{exact seq for S-gps} becomes
\[
1 \longrightarrow \mathbf{\mu_n}(\CC) \longrightarrow S_{\vp, \scn} \longrightarrow S_{\vp} \longrightarrow 1.
\]

Since $G$ is an $F$--inner form of $G^*,$ we can fix a character $\zeta_{G}$ of ${Z}(\widehat{G}_{\scn})$ 
which corresponds to the inner form $G$ of $G^*$ via the Kottwitz isomorphism \cite[Theorem 1.2]{kot86}.
When $D=F$ (i.e., $G=G^*$) we have $\zeta_{G} = \mathbbm{1}.$ 

Consider the exact sequence
\[
1 \longrightarrow \CC^{\times} \longrightarrow  {\GL}_n(\CC) \overset{pr}{\longrightarrow} {\PGL}_n(\CC) = \widehat{G} \longrightarrow 1.
\]
By the argument in Section \ref{thm by Labesse}, we have an 
$L$-parameter $\widetilde{\vp}$ for $\tG$
\[
\widetilde{\vp} : W_F \times {\SL}_n(\CC) \rightarrow {\GL}_n(\CC)
\]
such that ${pr} \circ \widetilde{\vp}=\vp$ (also see \cite{weil74, he80}).
By the LLC for $\tG$ \cite{hs11}, 
we have a unique irreducible representation $\ts \in \Irr(\tG)$ associated to the $L$-parameter $\widetilde{\vp}$ and the $L$-packet $\Pi_{\vp}(G)$ equals the set $\Pi_{\ts}(G)$ defined in Section \ref{results gelbart-knapp-tadic-hiraga-saito}.
By \cite[Lemma 12.6]{hs11} and \cite[Section 3]{abps13}, 
there is a homomorphism 
$\Lambda_{SL_n} : \cS_{\vp, \scn} \rightarrow \mcA(\ts)$ 
(unique up to one-dimensional characters of $\cS_{\vp}$) making the following diagram commute: 

\begin{equation} \label{a diagram}
\begin{CD}
1 @>>> \widehat Z_{\vp, \scn}  @>>> \cS_{\vp, \scn} @>>> \cS_{\vp}  @>>> 1 \\
@. @VV{\zeta_{G}}V @VV{\Lambda_{SL_n}}V @VV{\s}V @.\\
1 @>>> \CC^\times @>>> \mcA(\ts) @>>> I(\ts) @>>> 1.
\end{CD} 
\end{equation}
Combining \eqref{useful decomp} and \eqref{a diagram}, we have the following decomposition
\begin{equation}  \label{crucial decomp}
V_{\ts} ~ ~ \s \bigoplus_{\si \in \Pi_{\vp}(G)} \rho_{\si} \boxtimes  \si \quad
=
\bigoplus_{\rho \in \Irr(\cS_{\vp, \scn}, \zeta_{G})} \rho \boxtimes  \si_{\rho}
\end{equation}
as representations of $\cS_{\vp, \scn} \times G(F).$
Here, $\rho_{\sigma} \in \Irr(\cS_{\vp, \scn}, \zeta_{G})$ is given by $\xi_{\si} \circ \Lambda_{SL_n}$ with $\xi_{\si} \in \Irr(\mcA(\ts), \id),$ and $\si_{\rho} \in \Pi_{\vp}(G) $ denotes the image of $\rho \in \Irr(\cS_{\vp, \scn}, \zeta_{G})$ via the bijection between $\Pi_{\vp}(G)$ and $\Irr(\cS_{\vp, \scn}, \zeta_{G}).$

\section{Local Langlands Correspondence for \texorpdfstring{${\GSpin_4}$}{GSpin4} and its inner forms}  \label{LLC4}

In this section we establish the LLC for $\GSpin_4$ and all its non quasi-split $F$-inner forms. 

\subsection{Construction of $L$-packets of $\GSpin_4$ and Its Inner Forms} \label{const of L-packet for GSpin4}

It follows from the arguments in Section \ref{results in rest} on restriction and the group structure \eqref{derived of 4} that 
given $\sigma \in \Irr(\GSpin_4),$ there is a lifting $\ts \in \Irr({\GL}_2 \times {\GL}_2)$ such that 
\[
\sigma \hookrightarrow {\Res}_{{\GSpin}_{4}}^{\GL_2 \times \GL_2}(\ts).
\]
By the LLC for $GL_n$ \cite{ht01, he00, scholze13}, we have a unique $\tvp_{\ts} \in \Phi({\GL}_2 \times {\GL}_2)$ corresponding to the representation $\ts.$ 
We now define a map 
\begin{eqnarray} \label{L map for GSpin4}
{\L}_{4} : {\Irr}({\GSpin}_{4}) & \longrightarrow & \Phi({\GSpin}_{4}) \\
\sigma & \longmapsto & pr_4 \circ \tvp_{\ts}.  \nonumber
\end{eqnarray} 
Note that $\L_{4}$ does not depend on the choice of the lifting $\ts.$
Indeed, if $\ts' \in \Irr(\GL_2 \times \GL_2)$ is another lifting, it follows from Proposition \ref{pro for lifting} and \eqref{GSpin4-F} that 
$\ts' \s \ts \otimes \chi$ for some quasi-character $\chi$ on
\[
({\GL}_2(F) \times {\GL}_2(F))/{\GSpin}_4(F) \s F^{\times}. 
\]
Morover, 
\[
F^{\times} \s H^1(F, \CC^{\times}),
\]
where $\CC^{\times}$ is as in \eqref{ex seq L-gp for GSpin4}.
The LLC for $\GL_2 \times \GL_2$ maps
$\ts'$ to $\tvp_{\ts} \otimes \chi$ (by abuse of notation, employing $\chi$ for both the quasi-character and 
its parameter via Local Class Field Theory). 
Since $pr_4(\tvp_{\ts} \otimes \chi) = pr_4(\tvp_{\ts})$ by \eqref{ex seq L-gp for GSpin4}, the map $\L_{4}$ is well-defined.

Moreover, we note that $\L_4$ is surjective.
Indeed, by Labesse's Theorem in Section \ref{thm by Labesse}, $\vp \in \Phi(\GSpin_{4})$ 
can be lifted to some $\tvp \in \Phi(GL_2 \times \GL_2).$ 
We then obtain $\ts \in \Irr(GL_2 \times \GL_2)$ via the LLC for $\GL_2 \times \GL_2.$   
Thus, any irreducible constituent in the restriction ${\Res}_{{\GSpin}_{4}}^{\GL_2 \times \GL_2}(\ts)$ has the image $\vp$ via the map $\L_4.$

As expected, for each $\vp \in \Phi({\GSpin}_{4}),$ we define the $L$-packet $\Pi_{\vp}({\GSpin}_{4})$ 
as the set of all inequivalent irreducible constituents of $\ts$
\begin{equation} \label{def of L-packet for GSpin4}
\Pi_{\vp}({\GSpin}_{4}):=\Pi_{\ts}({\GSpin}_4) = 
\left\{ \sigma \hookrightarrow {\Res}_{{\GSpin}_{4}}^{\GL_2 \times \GL_2}(\ts) \right\} \slash \s,
\end{equation}
where $\ts$ is the unique member in $\Pi_{\tvp}(\GL_2 \times \GL_2)$ and 
$\tvp \in \Phi(\GL_2 \times \GL_2)$ is such that $pr_{4} \circ \tvp=\vp.$ 
By the LLC for $\GL_2$ and Proposition \ref{pro for lifting}, the fiber does not depends on the choice of $\tvp.$

We define the $L$-packets for the non quasi-split inner forms similarly.  Using the group structure described in Section \ref{gp structure},
given $\sigma^{2,1}_4 \in \Irr({\GSpin}^{2,1}_4),$ there is a lifting $\ts^{2,1}_4 \in \Irr({\GL}_2 \times {\GL}_1(D))$ such that 
\[
\sigma^{2,1}_4 \hookrightarrow {\Res}_{{\GSpin}^{2,1}_4}^{{\GL}_2 \times {\GL}_1(D)}(\ts^{2,1}_4).
\]
Combining the LLC for $\GL_2$ and $\GL_1(D)$ \cite{hs11},  
we have a unique $\tvp_{\ts^{2,1}_4} \in \Phi({\GL}_2 \times {\GL}_1(D))$ corresponding to the representation $\ts^{2,1}_4.$
We thus define the map  
\begin{eqnarray} \label{L map for inner 1 GSpin4}
{\L}^{2,1}_4 : {\Irr}({\GSpin}^{2,1}_4) &\longrightarrow & \Phi({\GSpin}^{2,1}_4) \\
\sigma^{2,1}_4 & \longmapsto & pr_4 \circ \tvp_{\ts^{2,1}_4}. \nonumber
\end{eqnarray}
Again, it follows from the LLC for $\GL_2$ and $\GL_1(D)$ that this map is well-defined and surjective.

Likewise, for the other $F$-inner form $\GSpin^{1,1}_4$ of $\GSpin_4,$ we have a well-defined and surjective map
\begin{eqnarray} \label{L map for inner 2 GSpin4}
{\L}^{1,1}_4 : {\Irr}({\GSpin}^{1,1}_4) & \longrightarrow & \Phi({\GSpin}^{1,1}_4) \\
\sigma^{1,1}_4 & \longmapsto & pr_4 \circ \tvp_{\ts^{1,1}_4}. \nonumber 
\end{eqnarray}

We similarly define $L$-packets 
\begin{equation} \label{def of inner 1 GSpin4}
\Pi_{\vp}({\GSpin}^{2,1}_4) = \Pi_{\ts^{2,1}_4}({\GSpin}^{2,1}_4), \quad \vp \in \Phi({\GSpin}^{2,1}_4)
\end{equation}
and 
\begin{equation} \label{def of inner 2 GSpin4}
\Pi_{\vp}({\GSpin}^{1,1}_4) = \Pi_{\ts^{1,1}_4}({\GSpin}^{1,1}_4), \quad \vp \in \Phi({\GSpin}^{1,1}_4). 
\end{equation}
Again, these $L$-packet do not depend on the choice of $\tvp$ for similar reasons.

\subsection{Internal Structure of $L$-packets of $\GSpin_4$ and Its Inner Forms} \label{para for GSpin4}
In this section we continue to employ the notation in Section \ref{general LLC}. 
For simplicity of notation, we shall write ${\GSpin}_\sharp$ for the split
${\GSpin}_{4},$ and its non quasi-split $F$-inner forms ${\GSpin}^{2,1}_4$ and ${\GSpin}^{1,1}_4.$
Likewise, we shall write ${\SL}_\sharp$ and ${\GL}_\sharp$ for the corresponding groups in 
\eqref{derived of 4}, \eqref{derived of 4 inner 1}, and \eqref{derived of 4 inner 2} so that
we have
\begin{equation} \label{sharp}
{\SL}_\sharp \subset {\GSpin}_\sharp \subset  {\GL}_\sharp 
\end{equation}
in all cases.  From Section \ref{L-groups}, we recall that 
\begin{align*}
(\widehat{{\GSpin}_\sharp})_{\ad} &= {\PSO}_4(\CC) \s {\PGL}_2(\CC) \times {\PGL}_2(\CC), \\ 
(\widehat{{\GSpin}_\sharp})_{\scn} &={\Spin}_4(\CC) \s {\SL}_2(\CC) \times {\SL}_2(\CC), \\
Z((\widehat{{\GSpin}_\sharp})_{\scn}) &= Z((\widehat{{\GSpin}_\sharp})_{\scn})^{\Gamma} \s \mu_2(\CC) \times \mu_2(\CC).
\end{align*}

Let $\vp \in \Phi({\GSpin}_\sharp)$ be given.
We fix a lifting $\tvp \in \Phi({\GL}_\sharp)$ via the surjective map 
$\widehat{{\GL}_\sharp}  \longrightarrow \widehat{{\GSpin}_\sharp}$ (cf. Theorem \ref{thm by Labesse}). 
With notation as in Section \ref{general LLC}, we have
\begin{align*}
S_{\vp} & \subset  {\PSL}_2(\CC)\times {\PSL}_2(\CC),  \\
S_{\vp, \scn} &\subset {\SL}_2(\CC)\times {\SL}_2(\CC).
\end{align*}
We then have (cf. \eqref{central ext}) a central extension 
\begin{equation} \label{central ext GSpin4}
1 \longrightarrow \widehat Z_{\vp, \scn}  \longrightarrow 
\cS_{\vp, \scn} \longrightarrow \cS_{\vp} \longrightarrow 1.
\end{equation}
Let $\zeta_{4},$ $\zeta^{2,1}_4,$ and $\zeta^{1,1}_4$ be characters on $Z((\widehat{{\GSpin}_\sharp})_{\scn})$ 
which respectively correspond to $\GSpin_{4},$ $\GSpin^{2,1}_4,$ and $\GSpin^{1,1}_4$ via the Kottwitz isomorphism \cite[Theorem 1.2]{kot86}. 
Note that 
\[\zeta_{4} = \mathbbm{1}, \quad \zeta^{2,1}_4 = \mathbbm{1} \times \sgn, \quad \mbox{ and } \quad \zeta^{1,1}_4 = \sgn \times \sgn, 
\]
where $\sgn$ is the non-trivial character of $\mu_2(\CC).$

\begin{thm} \label{1-1 for GSpin4}
Given an $L$-parameter $\vp \in \Phi(\GSpin_{\sharp}),$ 
there is a one-to-one bijection 
\begin{eqnarray*}
\Pi_{\vp}({\GSpin}_{\sharp}) & \, \overset{1-1}{\longleftrightarrow} \,  & \Irr(\cS_{\vp, \scn}, \zeta_{\sharp}), \\
 \sigma &\mapsto & \rho_{\sigma},
\end{eqnarray*}
such that we have the following decomposition
\[
V_{\ts} ~ ~ \s \bigoplus_{\sigma \in \Pi_{\vp}({\GSpin}_{\sharp})} \rho_{\sigma} \boxtimes  \si
\]
as representations of the direct product $\cS_{\vp, \scn} \times {\GSpin}_{\sharp}(F),$
where $\ts \in \Pi_{\tvp}(\GL_{\sharp})$ is an extension of $\si \in \Pi_{\vp}({\GSpin}_{\sharp})$ to $\GL_{\sharp}(F)$ 
as in Section \ref{results in rest} and $\tvp \in \Phi(\GL_{\sharp})$ is a lifting of $\vp \in \Phi(\GSpin_{\sharp}).$ 
Here, $\zeta_{\sharp} \in \left\{ \zeta_{4}, \zeta^{2,1}_4, \zeta^{1,1}_4 \right\}$ according to which inner form 
$\GSpin_{\sharp}$ is.  
\end{thm}

\begin{proof}
We follow the ideas in Section \ref{for sl(m,D)} and \cite[Theorem 5.4.1]{chaoli}.
Given $\vp \in \Phi(\GSpin_{\sharp}),$ we choose a lifting $\tvp \in \Phi(\GL_{\sharp})$ 
and obtain the projection $\bar{\vp} \in \Phi(\SL_{\sharp})$ in the following commutative diagram
\begin{equation} 
\label{projection parameters 4}
\xymatrix{
& {\quad \quad \quad \quad \quad \quad} \widehat{{\GL}_{\sharp}} = {\GL}_2(\CC) \times  {\GL}_2(\CC) \ar@{->>}[d]^{{~~}pr_4}\\
W_F \times {\SL}_2(\CC) \ar@{->}[ur]^{\tvp} \ar@{->}[r]^{\vp}
\ar@{->}[dr]_{\bar{\vp}} 
& {\quad} \widehat{{\GSpin}_{\sharp}} = {\GSO}_4(\CC)  \ar@{->>}^{{~~}\bar{pr}}[d] 
\\
& {\quad \quad \quad \quad \quad \quad \quad} \widehat{{\SL}_{\sharp}} = {\PGL}_2(\CC) \times  {\PGL}_2(\CC) ~ .
}
\end{equation}
We then have $\ts \in \Pi_{\tvp}(\GL_{\sharp})$ which is an extension of $\sigma \in \Pi_{\vp}(\GSpin_{\sharp}).$
In addition to \eqref{ex seq L-gp for GSpin4}, we also have 
\[
1 \longrightarrow \CC^{\times} \times \CC^{\times} \longrightarrow {\GL}_2(\CC) \times 
{\GL}_2(\CC) \overset{\bar{pr} \circ pr_4}{\longrightarrow} {\PGL}_2(\CC) \times  {\PGL}_2(\CC) \longrightarrow 1 
\]
Considering the kernels of the projections $pr_4$ and $\bar{pr} \circ pr_4,$
we set
\[
X^{{\GSpin}_{\sharp}}(\tvp) := \{ a \in H^1(W_F, \CC^{\times}) : a \tvp \s \tvp \}
\]
\[
X^{{\SL}_{\sharp}}(\tvp) := \{ a \in H^1(W_F, \CC^{\times} \times \CC^{\times}) : a \tvp \s \tvp \}.
\]
Moreover, by \eqref{convenient exact sequence GSpin4} and its analogues for the two non quasi-split $F$-inner forms, we have 
\[
{\GL}_{\sharp}(F) / {\GSpin}_{\sharp}(F) \s F^{\times}.
\]
As an easy consequence of Galois cohomology, we also have
\[
{\GL}_{\sharp}(F) / {\SL}_{\sharp}(F) \s F^{\times} \times F^{\times}.
\]
Set
\[
I^{{\GSpin}_{\sharp}}(\ts) := \{ \chi \in (F^{\times})^D \s ({\GL}_{\sharp}(F) / {\GSpin}_{\sharp}(F))^D : \ts \chi  \s \ts \}
\]
\[
I^{{\SL}_{\sharp}}(\ts) := \{ \chi \in (F^{\times})^D \times (F^{\times})^D \s ({\GL}_{\sharp}(F) / {\SL}_{\sharp}(F))^D : \ts \chi  \s \ts \}.
\]
\begin{rem} \label{rem on characters}
Since any character on $\GL_n(F)$  (respectively, $\GL_m(D)$) is of the form $\chi \circ \det$ (respectively, $\chi \circ \Nrd$) 
for some character $\chi$ on $F^{\times}$ (see \cite[Section 53.5]{bh06}), 
we often make no distinction between $\chi$ and $\chi \circ \det$ (respectively, $\chi \circ \Nrd$).
Moreover, we note that $\chi \in ({\GL}_{\sharp}(F))^D$ can be written as $\tchi_1 \boxtimes \tchi_2,$ where $\tchi_i$ 
with $i=1,2,$ is a character of $\GL_2(F)$ or $D^{\times}$ depending on the type of ${\GL}_{\sharp}.$
\end{rem}

It follows from the definitions that 
\begin{equation} \label{subset X I 4}
X^{{\GSpin}_{\sharp}}(\tvp) \subset X^{{\SL}_{\sharp}}(\tvp) \quad \mbox{ and } \quad 
I^{{\GSpin}_{\sharp}}(\ts) \subset I^{{\SL}_{\sharp}}(\ts). 
\end{equation}
We have $(F^{\times})^D \s H^1(W_F, \CC^{\times})$ by the local class field theory. 
It is also immediate from the LLC for $GL_n$ that
\[
X^{{\GSpin}_{\sharp}}(\tvp) \s I^{{\GSpin}_{\sharp}}(\ts) \quad \mbox{ and } \quad X^{{\SL}_{\sharp}}(\tvp) \s I^{{\SL}_{\sharp}}(\ts)
\]
as groups of characters. 
Now, with the component group notation $\cS$ of Section \ref{general LLC}, we claim that
\begin{equation} \label{crutial iso for gspin4}
I^{{\GSpin}_{\sharp}}(\ts) \s \cS_{\vp}.
\end{equation}
Indeed, by \cite[Lemma 5.3.4]{chaoli} and the above arguments, this claim follows from 
\[
\cS_{\vp} \s X^{{\GSpin}_{\sharp}}(\tvp),
\]
since $\cS_{\tvp}$ is always trivial by the LLC for ${\GL}_{\sharp}.$  
Notice that $\tvp$ and $\vp$ here are respectively $\phi$ and $\phi^{\sharp}$ in \cite[Lemma 5.3.4]{chaoli}.
Thus, due to \eqref{a diagram}, \eqref{subset X I 4}, and\eqref{crutial iso for gspin4} we have
\begin{equation} \label{subset cS 4}
\cS_{\vp} \subset \cS_{\bar\vp}.
\end{equation}

Since the centralizer $C_{\bar\vp}$ (in $\widehat{{\SL}}_{\sharp}$) is equal to the image of the disjoint union
\[
\coprod_{\nu \in {\Hom}(W_F, \CC^{\times})} 
\left\{ h \in {\GSO}_4(\CC) : h \vp(w)h^{-1} \vp(w)^{-1} = \nu(w) \right\}
\]
from the exact sequence
\[
1 \longrightarrow \CC^{\times} \longrightarrow {\GSO}_4(\CC) \overset{\bar{pr} }{\longrightarrow} 
\widehat{{\SL}_{\sharp}} = {\PGL}_2(\CC) \times  {\PGL}_2(\CC) \longrightarrow 1,
\]
we have 
\[
S_{\vp} \subset C_{\bar\vp} = S_{\bar\vp}.
\]
So, the pre-images in $\SL_2(\CC) \times \SL_2(\CC) $ via the isogeny 
$\SL_2(\CC) \times \SL_2(\CC) \twoheadrightarrow \PGL_2(\CC) \times \PGL_2(\CC)$ also satisfy
\[
S_{\vp, \scn}
\subset
S_{\bar\vp, \scn}
\subset {\SL}_2(\CC) \times {\SL}_2(\CC).
\]
This provides the inclusion of the identity components
\begin{equation} \label{subset S sc 4}
S_{\vp, \scn}^{\circ} \subset 
S_{\bar\vp, \scn}^{\circ},
\end{equation}
and by the definition $\widehat Z_{\vp, \scn} := Z(\widehat G_{\scn}) / (Z(\widehat G_{\scn}) \cap S_{\vp, \scn}^{\circ})$ 
of Section \ref{general LLC} 
we have the surjection
\begin{equation} \label{subset Z 4}
\widehat Z_{\vp, \scn} \twoheadrightarrow \widehat Z_{\bar{\vp}, \scn}. 
\end{equation}
Combining \eqref{subset cS 4}, \eqref{subset S sc 4}, and \eqref{subset Z 4}, 
we have the following commutative diagram of component groups: 
\begin{equation} \label{pre pre diagram gspin4}
\begin{CD}
1 @>>>  \widehat Z_{\vp, \scn}  @>>> \cS_{\vp, \scn} @>>> \cS_{\varphi}  @>>> 1 \\ 
@. @VV{\twoheadrightarrow}V @VV{\cap}V @VV{\cap}V @.\\
1 @>>>  \widehat Z_{\bar\vp, \scn}  @>>> \cS_{\bar\vp, \scn} @>>> \cS_{\bar\vp}  @>>> 1.
\end{CD} 
\end{equation}
We apply the Hiraga-Saito homomorphism $\Lambda_{\SL_n}$ of \eqref{a diagram} 
in the case of our $\SL_{\sharp},$ which we denote by $\Lambda_{\SL_2 \times \SL_2}.$  
The restriction
\[
\Lambda_{\sharp} := \Lambda_{\SL_2 \times \SL_2}\vert_{\cS_{\vp, \scn}} 
\] 
then gives the following commutative diagram: 
\begin{equation} \label{pre diagram gspin4}
\begin{CD}
1 @>>> \widehat Z_{\vp, \scn}  @>>> \cS_{\vp, \scn} @>>> \cS_{\vp}  @>>> 1 \\
@. @VV{\zeta_{\sharp}}V @VV{\Lambda_{\sharp}}V @VV{\cap}V @.\\
1 @>>> \CC^\times @>>> \mcA(\ts) @>>> I^{\SL_{\sharp}}(\ts) @>>> 1.
\end{CD} 
\end{equation}
Note that $\zeta_{\sharp}$ is identified with $\zeta_G$ in \eqref{a diagram} as the character on $\mu_2(\CC) \times \mu_2(\CC)$ 
since both are determined according to $\SL_\sharp.$

Similar to $\mcA(\ts)$ (cf. Section \ref{results in rest}), we write $\mcA^{\GSpin_{\sharp}}(\ts)$ 
for the subgroup of ${\Aut}_{\CC}(V_{\ts})$ generated by $\CC^\times$ and 
$\left\{I_{\chi} : \chi \in I^{\GSpin_{\sharp}}(\ts) \right\}.$   Hence, $\mcA^{\GSpin_{\sharp}}(\ts) \subset \mcA(\ts).$
By the definition of $\Lambda_{\SL_2\times\SL_2}$ in \eqref{a diagram} and the commutative diagram \eqref{pre pre diagram gspin4}, it is immediate that the image of $\Lambda_{\sharp}$ is 
$\mcA^{\GSpin_{\sharp}}(\ts).$
We thus have the following commutative diagram: 
\begin{equation} \label{diagram gspin4}
\begin{CD}
1 @>>> \widehat Z_{\vp, \scn}  @>>> \cS_{\vp, \scn} @>>> \cS_{\vp}  @>>> 1 \\
@. @VV{\zeta_{\sharp}}V @VV{\Lambda_{\sharp}}V @VV{\s}V @.\\
1 @>>> \CC^\times @>>> \mcA^{\GSpin_{\sharp}}(\ts) @>>> I^{\GSpin_{\sharp}}(\ts) @>>> 1.
\end{CD} 
\end{equation}
The representation $\rho_{\sigma} \in \Irr(\cS_{\vp, \scn}, \zeta_{\sharp})$ is defined by $\xi_{\si} \circ \Lambda_{\sharp},$ 
where $\xi_{\si} \in \Irr(\mcA^{\GSpin_\sharp}(\ts), \id)$ is the character as in the decomposition \eqref{useful decomp}.
Therefore, arguments of Section \ref{results in rest}, 
our construction of $L$-packets $\Pi_{\vp}(\GSpin_\sharp)$ in Section \ref{const of L-packet for GSpin4} 
and diagram \eqref{diagram gspin4} above give 
\begin{equation} \label{imp bijection gspin4}
\Irr(\cS_{\vp, \scn}, \zeta_{\sharp}) \, \overset{1-1}{\longleftrightarrow} \, 
\Irr(\mcA^{\GSpin_{\sharp}}(\ts), \id) \, \overset{1-1}{\longleftrightarrow} 
\Pi_{\vp}({\GSpin}_{\sharp}). 
\end{equation} 
We also have the following decomposition
\[
V_{\ts} ~ ~ \s \bigoplus_{\si \in \Pi_{\vp}({\GSpin_\sharp})} \rho_{\si} \boxtimes  \si 
~ ~  
=
\bigoplus_{\rho \in \Irr(\cS_{\vp, \scn}, \zeta_{\sharp})} \rho \boxtimes  \si_{\rho},
\]
where $\si_{\rho}$ denotes the image of $\rho$ via the bijection between $\Pi_{\vp}(\GSpin_\sharp)$ 
and $\Irr(\cS_{\vp, \scn}, \zeta_{\sharp}).$ Hence, the proof of Theorem \ref{1-1 for GSpin4} is complete.
\end{proof}
\begin{rem} 
Similar to $\Lambda_{\SL_2\times\SL_2}$ in Section \ref{results in rest},  
since $\Lambda_{\sharp}$ is unique up to 
$\operatorname{Hom}(I^{\GSpin_\sharp}(\ts), \CC^{\times}) \s \operatorname{Hom}(\cS_{\vp}, \CC^{\times}),$ 
the same is true for the bijection \eqref{imp bijection gspin4}.
\end{rem}
\begin{rem}
Given $\vp \in \Phi(\GSpin_{4}),$ by Theorem \ref{1-1 for GSpin4}, we have a one-to-one bijection 
\[
\Pi_{\vp}({\GSpin}_{4}) \cup \Pi_{\vp}({\GSpin}^{2,1}_4) \cup 
\Pi_{\vp}({\GSpin}_4^{1,2}) \cup \Pi_{\vp}({{\GSpin}^{1,1}_4}) 
\overset{1-1}{\longleftrightarrow} {\Irr}(\cS_{\vp, \scn}),
\] 
where ${\GSpin}_4^{1,2} \s \{ (g_1, g_2) \in (\GL_1(D) \times \GL_2) : \Nrd(g_1) = \det g_2 \},$ which is isomorphic to ${\GSpin}^{2,1}_4.$
\end{rem}
%

\subsection{$L$-packet Sizes for $\GSpin_4$ and Its Inner Forms} \label{size of L-pacekt gspin4}
In this subsection, we describe the group $I^{\GSpin_4}(\ts)$ for $\ts \in \Irr(\GL_2 \times \GL_2)$ case by case, discuss the group structure of $\cS_{\vp, \scn},$ and provide all the sizes of $L$-packets for $\GSpin_{4},$ $\GSpin^{2,1}_4,$ and $\GSpin^{1,1}_4$ using Theorem \ref{1-1 for GSpin4}.
Let us first give the possible cardinalities for the $L$-packets of $\GSpin_4$ and its inner forms using Galois cohomology. 
\begin{pro} \label{pro size 4}
Let $\Pi_{\vp}(\GSpin_\sharp)$ be an $L$-packet associated to $\vp \in \Phi(\GSpin_\sharp).$  
Then we have
\[
\left| \Pi_{\vp}({\GSpin}_\sharp) \right| 
\; \Big{|} \; 
\left| F^{\times}/(F^{\times})^2 \right|,
\]
which yields the possible cardinalities 
\[
\left| \Pi_{\vp}({\GSpin}_\sharp) \right| 
= \left\{ 
\begin{array}{l l}
    1, ~ 2, ~ 4, & \: \text{if} ~ p \neq 2, \\

    1, ~ 2, ~ 4, ~ 8, & \: \text{if} ~ p= 2. \\
  \end{array}
  \right.
\]
\end{pro}
\begin{proof}
We follow the idea of a similar result in the case of $\Sp_4$ in \cite{chgo12}.
The exact sequence of algebraic groups 
\[
1 \longrightarrow Z({\GSpin}_\sharp) \longrightarrow Z({\GL}_\sharp) \times {\GSpin}_\sharp \longrightarrow {\GL}_\sharp \longrightarrow 1,
\] 
where the second map is given by multiplication (considering $\GSpin_\sharp \subset \GL_\sharp$),  
gives a long exact sequence
\[
\cdots \longrightarrow (F^{\times} \times F^{\times}) \times {\GSpin}_\sharp(F) \longrightarrow {\GL}_\sharp(F) \longrightarrow H^1(F, Z({\GSpin}_\sharp)) \longrightarrow H^1(F, Z({\GL}_\sharp) \times {\GSpin}_\sharp) \longrightarrow 1.
\]
Since $H^1(F, Z({\GL}_\sharp) \times {\GSpin}_\sharp) = 1$ by Lemma \ref{GC for GSpin} and \cite[Lemmas 2.8]{pr94},  
we have 
\[
{\GL}_\sharp(F) / \left( (F^{\times} \times F^{\times}) \times {\GSpin}_\sharp(F) \right) \hookrightarrow  H^1(F, Z({\GSpin}_\sharp)).
\]
Also, the exact sequence
\[
1 \longrightarrow Z({\GSpin}_\sharp)^{\circ} \longrightarrow Z({\GSpin}_\sharp) \longrightarrow \pi_0(Z({\GSpin}_\sharp)) \longrightarrow 1 
\]
gives 
\[
H^1(F, Z({\GSpin}_\sharp)) \hookrightarrow H^1(F, \pi_0( Z({\GSpin}_\sharp))),
\]
since, by \cite[Proposition 2.3]{ash06duke}, $Z({\GSpin}_\sharp)^{\circ} \s \GL_1$  and $H^1(F, Z({\GSpin}_\sharp)^{\circ})=1.$
Combining the above with \eqref{centers}, we have
\begin{equation} \label{coker imbeding 4}
{\GL}_\sharp(F) / \left( (F^{\times} \times F^{\times}) \times {\GSpin}_\sharp(F) \right) 
\hookrightarrow  H^1(F, \pi_0( Z({\GSpin}_\sharp))) \s H^1(F, \ZZ/2\ZZ) \s F^{\times}/(F^{\times})^2.
\end{equation}

On the other hand, we know from Section \ref{results in rest} that 
\[
\Pi_{\vp}({\GSpin}_\sharp) \, \overset{1-1}{\longleftrightarrow} \, 
{\GL}_\sharp(F) / {\GL}_\sharp(F)_{\ts}
\hookrightarrow
{\GL}_\sharp(F) / \left( (F^{\times} \times F^{\times}) \times {\GSpin}_\sharp(F) \right),
\]
where $\ts \in \Irr(\GL_\sharp)$ corresponds to a lifting $\tvp \in \Phi(\GL_\sharp),$ via the LLC for $GL_n$ and its inner forms, of $\vp \in \Phi(\GSpin_\sharp).$ 
Thus the proof is complete. 
\end{proof}

In what follows, we give an explicit description of the group $I^{\GSpin_4}(\ts)$ for $\ts \in \Irr(\GL_2 \times \GL_2)$ case by case, and show that 
among the possible cardinalities in Proposition \ref{pro size 4}, only 1, 2, and 4 do indeed occur (see Remarks \ref{last rem for card GSpin4} - \ref{rem 2 for I(GSpin4)} and \eqref{size L-packets for GSpin4 1} - \eqref{size L-packets for GSpin4 2}). 
To this end, 
we verify that $I^{\GSpin_4}(\ts) \s 1, \ZZ/2\ZZ,$ or $(\ZZ/2\ZZ)^2,$ which yields $\cS_{\vp} \s 1, ~~ \ZZ/2\ZZ,$ or $(\ZZ/2\ZZ)^2$ by \eqref{crutial iso for gspin4}. 
Thus, due to Theorem \ref{1-1 for GSpin4}, the size of an $L$-packet for $\GSpin_4$ is either 1, 2, or 4.
Moreover, we show that the central extension $\cS_{\vp, \scn}$ is isomorphic to $(\ZZ/2\ZZ)^2, (\ZZ/2\ZZ)^3$, an abelian subgroup of order 16 in $Q_8 \times Q_8,$ or $Q_8 \times \ZZ/2\ZZ,$ for irreducible $\vp \in \Phi(\GSpin_4),$ and is always abelian for reducible $\vp \in \Phi(\GSpin_4).$
Here $Q_8$ is the quaternion group of order $8.$ 
Due to Theorem \ref{1-1 for GSpin4}, these facts about $\cS_{\vp, \scn}$ imply that the size of an $L$-packet for non-split inner forms $\GSpin^{2,1}_4$ and $\GSpin^{1,1}_4$ is either 1, 2, or 4.

To describe the group $I^{\GSpin_4}(\ts),$ given $\ts \in \Irr(\GL_2 \times \GL_2),$ we set $\ts = \ts_1 \boxtimes \ts_2$ with $\ts_1, \ts_2 \in \Irr(\GL_2).$ 
Due to Remark \ref{rem on characters}, we note that $\chi \in ({\GL}_{2}(F) \times {\GL}_{2}(F))^D$ is decomposed into $\tchi_1 \boxtimes \tchi_2,$ where $\tchi_i$ with $i=1,2$ is a character of $\GL_2(F),$ and we identify $\tchi_i$ and $\tchi_i \circ \det,$ since any character on $\GL_n(F)$ is of the form $\tchi \circ \det$ for some character $\tchi$ on $F^{\times}.$

\begin{lm} \label{lemma form of characters for GSpin4}
Let $\chi \in I^{\GSpin_4}(\ts)$ be given. Then,
$\chi$ is of the form 
\[
\tchi \boxtimes \tchi^{-1},
\]
where $\tchi \in (F^{\times})^D.$  
\end{lm}
\begin{proof}
Since $\chi = \tchi_1 \boxtimes \tchi_2$ as above and $\chi$ is trivial on $\GSpin_4(F),$ by the structure of 
$\GSpin_4(F)$ in \eqref{convenient description of GSpin4(F)}, we have 
\[
\chi(g_1, g_2) = \tchi_1(\det g_1) \boxtimes \tchi_2 (\det g_2)=\tchi_1(\det g_1) \tchi_2 (\det g_2) =1
\]
for all $(g_1,g_2) \in \GL_2(F) \times \GL_2(F)$ with $\det g_1 = \det g_2.$  
Since the determinant map $\det: \GL_2(F) \rightarrow F^{\times}$ is surjective, $\tchi_1 \tchi_2(x)$ must be trivial for all $x \in F^{\times}.$ Thus, $\tchi_1$ and $\tchi_2$  are inverse each other.
\end{proof}

\begin{pro} \label{form of characters for I(GSpin4)}
We have
\[ 
I^{\GSpin_4}(\ts) = \left\{ \begin{array}{lll}
         I^{{\SL}_2}(\ts_1), & 	\mbox{if $\ts_2 \s \ts_1\teta$ for some $\teta \in (F^{\times})^D$}; \\
        I^{{\SL}_2}(\ts_1) \cap I^{{\SL}_2}(\ts_2), & \mbox{if $\ts_2 \not\s \ts_1\teta$ for any $\teta \in (F^{\times})^D$}.
				\end{array} \right.
\] 
\end{pro}
\begin{proof}
Since $\ts = \ts_1 \boxtimes \ts_2$ with $\ts_1, \ts_2 \in \Irr(\GL_2),$ and by Lemma \ref{lemma form of characters for GSpin4}, we have
\begin{equation} \label{iff for characters GSpin4}
\ts \chi \s \ts 
\;  \Longleftrightarrow \;
\ts_1 \tchi \boxtimes \ts_2 \tchi^{-1} \s \ts_1 \boxtimes \ts_2
\;  \Longleftrightarrow \;
\ts_i \tchi \s \ts_i, \;  i=1,2.
\end{equation}
This shows that $I^{\GSpin_4}(\ts) = I^{\SL_2}(\ts_1) \cap I^{\SL_2}(\ts_2).$ 
In particular, if $\ts_2 \s \ts_1\teta$ for some $\teta \in (F^{\times})^D,$ 
then $I^{\SL_2}(\ts_1) = I^{\SL_2}(\ts_2).$
Thus we have, by \eqref{iff for characters GSpin4}, that 
$
I^{\GSpin_4}(\ts) = I^{\SL_2}(\ts_1). 
$
\end{proof}

\begin{rem}  \label{rem for case I II III}
By the LLC for $\SL_2$ (\cite{gk82}), we have $I^{\SL_2}(\ts_i) \s \pi_0(C_{\vp_i})$ 
for $\vp_i \in \Phi(\SL_2)$ corresponding to $\si_i \subset \Res^{\GL_2}_{\SL_2}(\ts_i).$   Recalling that 
$\PGL_2(\CC) \s \SO_3(\CC),$ it then follows from \cite[Corollary 6.6]{grossprasad92} that
$I^{\SL_2}(\ts_1)$ and $I^{\SL_2}(\ts_2)$ consist of quadratic characters. 
Then, $I^{\GSpin_4}(\ts)$ consists of quadratic characters of $F^{\times}.$
\end{rem}

Given $\vp \in \Phi(\GSpin_4),$ fix the lift 
\[ 
\tvp=\tvp_1 \otimes \tvp_2 \in \Phi(\GL_2\times \GL_2) 
\] 
with $\tvp_i \in \Phi(\GL_2)$ such that $\vp = pr_4 \circ \tvp,$ as in Section \ref{const of L-packet for GSpin4}.
Let 
\[ 
\ts = \ts_1 \boxtimes \ts_2 \in \Pi_{\tvp}(\GL_2\times \GL_2) 
\] 
be the unique member such that $\{\ts_i\} = \Pi_{\tvp_i}(\GL_2).$
Recall (see \cite[Section 5]{gtsp10}) that an irreducible $L$-parameter $\phi \in \Phi(\GL_2)$ 
is called \textit{primitive} if the restriction $\phi|_{W_F}$ is not of the form $\Ind_{W_E}^{W_F} \theta$ for some finite extension $E$ over $F.$ 
Moreover, $\phi$ is called \textit{dihedral} with respect to (w.r.t.) a quadratic extension $E$ over $F$ if 
$\phi|_{W_F} \s \Ind_{W_E}^{W_F} \theta$ or equivalently if 
$(\phi|_{W_F}) \otimes \omega_{E/F} \s (\phi|_{W_F})$ ($\Leftrightarrow \phi \otimes \omega_{E/F} \s \phi$) 
for a quadratic character $\omega_{E/F}$ corresponding to quadratic $E/F$ via the local class field theory.
A primitive representation exists only when $p$ divides $\dim \phi$ (\cite{koch77}). 
We can now make the following remarks.

\begin{rem} \label{last rem for card GSpin4}
If $\ts_2 \s \ts_1\teta$ for some $\teta \in (F^{\times})^D,$ 
we know that $\tvp_1$ is dihedral w.r.t. one quadratic character 
(respectively, primitive, dihedral w.r.t. three quadratic characters) if and only if $\tvp_2$ is dihedral 
(respectively, primitive, dihedral w.r.t. three quadratic characters).
\end{rem}

\begin{rem} \label{cor for I(GSpin4)}
Let $\vp \in \Phi(\GSpin_4)$ be irreducible.  Then $\tvp,$ $\tvp_1,$ and $\tvp_2$ are all irreducible.
Combining Proposition \ref{form of characters for I(GSpin4)}, Remarks \ref{rem for case I II III} and 
\ref{last rem for card GSpin4}, and \cite[Porposition 6.3]{gtsp10}, we conclude the following.
When $\ts_2 \s \ts_1\teta$ for some $\teta \in (F^{\times})^D,$ we have
\[ 
I^{\GSpin_4}(\ts) \s 
\left\{ 
\begin{array}{llll}
         \{1\}, & 
         \mbox{if $\tvp_1$ is primitive or non-trivial on $\SL_2(\CC)$ (hence so is $\tvp_2$)};\\
         \ZZ/2\ZZ, &  
         \mbox{if $\tvp_1$ is dihedral w.r.t. one quadratic extension (hence so is $\tvp_2$)};\\ 
         (\ZZ/2\ZZ)^2, &  
         \mbox{if $\tvp_1$ is dihedral w.r.t. three quadratic extensions (hence so is $\tvp_2$)}.
\end{array} \right. 
\]
When $\ts_2 \not\s \ts_1\teta$ for any $\teta \in (F^{\times})^D$, an analogous assertion holds, but it 
would depend on the individual parameters and not just whether they are primitive or dihedral.  In that case, we have  
\[ 
I^{\GSpin_4}(\ts) \s \{1\},~ \ZZ/2\ZZ, ~\text{or } (\ZZ/2\ZZ)^2.
\]
\end{rem}
\begin{rem} \label{rem 2 for I(GSpin4)}
When $\tvp_i$ is reducible, 
$\ts_i$ is either the Steinberg representation twisted by a character or an irreducibly induced representation from the Borel subgroup of $\GL_2.$ 
Since the number of irreducible constituents in $\Res_{\SL_2}^{\GL_2} (\ts_i)$ is $\leq 2,$ we have 
$I^{\SL_2}(\ts_i) \s \{1\}, ~\text{or } \ZZ/2\ZZ.$
This implies, by Proposition \ref{form of characters for I(GSpin4)}, that 
\[
I^{\GSpin_4}(\ts) \s \{1\}, ~ \text{or } \ZZ/2\ZZ.
\]
\end{rem}
We now describe the central extension $\cS_{\vp, \scn}$ for $\GSpin_4.$ 
For irreducible $\vp \in \Phi(\GSpin_4),$ from \eqref{crutial iso for gspin4} and Remark \ref{cor for I(GSpin4)}, we have $\cS_{\vp} \s 1, ~~ \ZZ/2\ZZ,$ or  $(\ZZ/2\ZZ)^2.$ Since $\vp$ is elliptic and 
$\widehat Z_{\vp, \scn}
 =  Z(\Spin_4(\CC))= \mu_2(\CC) \times \mu_2(\CC),$ the exact sequence \eqref{central ext GSpin4} becomes
\[
1 \longrightarrow \mu_2(\CC) \times \mu_2(\CC)   \longrightarrow 
\cS_{\vp, \scn} \subset {\Spin}_4(\CC) = {\SL}_2(\CC) \times {\SL}_2(\CC)  \longrightarrow \cS_{\vp} \subset  {\PSL}_2(\CC) \times {\PSL}_2(\CC) \longrightarrow 1.
\]
When $\cS_{\vp} \s 1,$ it is obvious that $\cS_{\vp, \scn} \s (\ZZ/2\ZZ)^2.$ 
When $\cS_{\vp} \s \ZZ/2\ZZ,$ 
since $(\ZZ/2\ZZ)^2$ is embedded in $Z(\cS_{\vp, \scn}),$
 $\cS_{\vp, \scn}$ is an abelian group of order 8. Since $\cS_{\vp, \scn} \subset \cS_{\bar{\vp}, \scn}$ from \eqref{pre pre diagram gspin4}, using the fact that the group $\cS_{\bar{\vp}, \scn} \s (\ZZ/2\ZZ)^4$ (see the proof of \cite[Proposition 6.5]{choiymulti}), we have $\cS_{\vp, \scn} \s (\ZZ/2\ZZ)^3.$ 
When $\cS_{\vp} \s (\ZZ/2\ZZ)^2,$ since $\cS_{\bar{\vp}, \scn} \s Q_8 \times Q_8$ from \textit{loc. cit.}, $\cS_{\vp, \scn}$ is a subgroup of order 16 in $Q_8 \times Q_8.$
If $\cS_{\vp, \scn}$ is non-abelian, since $Z(\cS_{\vp, \scn}) \s (\ZZ/2\ZZ)^2,$  we have $\cS_{\vp, \scn} \s Q_8 \times \ZZ/2\ZZ.$
It thus follows that $\cS_{\vp, \scn}$ is an abelian subgroup in $Q_8 \times Q_8$ of order 16, or $Q_8 \times \ZZ/2\ZZ.$ 
For reducible $\vp \in \Phi(\GSpin_4),$ \eqref{crutial iso for gspin4} and Remark \ref{rem 2 for I(GSpin4)} provide $\cS_{\vp} \s 1$ or $\ZZ/2\ZZ.$ Since $(\ZZ/2\ZZ)^2$ is embedded in $Z(\cS_{\vp, \scn}),$  $\cS_{\vp, \scn}$ must be abelian of order 8.

Using Theorem \ref{1-1 for GSpin4},  we describe the sizes of $L$-packets for $\GSpin_{4},$ $\GSpin^{2,1}_4,$ and $\GSpin^{1,1}_4.$  
Form the previous paragraph, we note that $\cS_{\vp, \scn}$ is either abelian or isomorphic to $Q_8 \times \ZZ/2\ZZ.$ 
Thus, we have
\begin{equation} \label{size L-packets for GSpin4 1}
\left| \Pi_{\vp}({\GSpin}_4) \right| = \left| \Pi_{\vp}({\GSpin}^{2,1}_4) \right|  = \left| \Pi_{\vp}({\GSpin}^{1,1}_4) \right| = \left| \cS_{\vp}  \right| = 1, ~~2, \text{or} ~~4,
\end{equation}
according to 
whether 
$\cS_{\vp} \s \{1\},~ \ZZ/2\ZZ, ~\text{or } (\ZZ/2\ZZ)^2,$
except when $\cS_{\vp, \scn} \s Q_8 \times \ZZ/2\ZZ,$ in which case we have
\begin{equation} \label{size L-packets for GSpin4 2}
\left| \Pi_{\vp}({\GSpin}_4) \right| = 4, ~~ \left| \Pi_{\vp}({\GSpin}^{2,1}_4) \right|  = \left| \Pi_{\vp}({\GSpin}^{1,1}_4) \right|  = 1.
\end{equation}
Furthermore, using \eqref{multi and characters}, the multiplicity in restriction from $\GL_{\sharp}$ to $\GSpin_\sharp$ is 1 except when $\cS_{\vp, \scn} \s Q_8 \times \ZZ/2\ZZ,$ in which case the multiplicity is 2.

We should mention that, due to Remarks \ref{cor for I(GSpin4)} and \ref{rem 2 for I(GSpin4)}, 
the case of $\cS_{\vp, \scn} \s Q_8 \times \ZZ/2\ZZ$ occurs only if two conditions hold
: $\ts_2 \s \ts_1\teta$ for some $\teta \in (F^{\times})^D,$ and $\tvp_1$ is dihedral w.r.t. three quadratic extensions (hence so is $\tvp_2$).

\subsection{Properties of $\L$-maps for $\GSpin_{4}$ and its inner forms} \label{properties for gspin4}

The $\L$-maps defined in Section \ref{const of L-packet for GSpin4} satisfy a number of natural and expected properties which we now verify.  
In what follows, let $\L_{\sharp} \in \left\{\L_4, \L_4^{2,1}, \L_4^{1,1} \right\}$ as the case may be. 

\begin{pro} \label{discreteness for gspin4}
A representation $\sigma_{\sharp} \in \Irr({\GSpin}_\sharp)$ is essentially square integrable 
if and only if its $L$-parameter $\vp_{\sigma_{\sharp}} := \L_{\sharp}(\sigma_{\sharp})$ does not factor through any proper Levi subgroup of ${\GSO}_4(\CC).$
\end{pro}
\begin{proof}
By the definition of $\L_{\sharp},$ the representation $\sigma_{\sharp}$ is an irreducible constituent of the restriction 
$\Res^{\GL_{\sharp}}_{{\GSpin}_\sharp} (\ts_{\sharp})$ for some $\ts_{\sharp} \in \Irr({\GL}_\sharp).$
As recalled in Remark \ref{rem in rest from to}, $\sigma_{\sharp}$ is essentially square integrable representation if and only if $\ts_{\sharp}$ is so.  
By the LLC for $\GL_\sharp$ and its inner forms \cite{ht01, he00, scholze13, hs11}, this is the case if and only if 
its parameter $\tvp_{\sigma_{\sharp}}:=\L(\ts_{\sharp})$ does not factor through any proper Levi subgroup of ${\GL}_\sharp(\CC).$  
Finally, this is the case if and only if $\vp_{\sigma_{\sharp}}$ does not because the projection $pr_4$ in \eqref{ex seq L-gp for GSpin4} 
respects Levi subgroups. 
\end{proof}
\begin{rem}
In the same way as in the proof of Proposition \ref{discreteness for gspin4}, we have that 
a given $\sigma_{\sharp} \in \Irr({\GSpin}_\sharp)$ is tempered if and only if 
the image of its $L$-parameter $\vp_{\sigma_{\sharp}} := \L_{\sharp}(\sigma_{\sharp})$ in ${\GSO}_4(\CC)$ is bounded. 
\end{rem}

Due to the fact that restriction of representations preserves the intertwining operator and 
the Plancherel measure (see \cite[Section 2.2]{choiy1}), our construction of the $L$-packets in Section \ref{const of L-packet for GSpin4} 
gives the following result.

\begin{pro} \label{pm for gspin4}
Let $\vp \in \Phi_{\disc}(\GSpin_\sharp)$ be given. For any $\sigma_1, \sigma_2 \in \Pi_{\vp}(\GSpin_\sharp),$ 
we have the equality of the Plancherel measures
\begin{equation} \label{equlity one}
\mu_M(\nu, \tau \boxtimes \si_1, w) = \mu_M(\nu, \tau \boxtimes \si_2, w),
\end{equation}
where $M$ is an $F$-Levi subgroup of an $F$-inner form of $\GSpin_{2n}$ of the form of the product of $\GSpin_\sharp$ 
and copies of $F$-inner forms of $\GL_{m_i},$ $\tau \boxtimes \si_1, \tau \boxtimes \si_2 \in \Pi_{\disc}(M),$ 
$\nu \in \mathfrak{a}^{*}_{M, \CC},$ and $w \in W_M$ with $^wM = M.$
Further, it is a consequence of the equality of the Plancherel measures that the Plancherel measure is also preserved between $F$-inner forms 
in the following sense.  Let $\GSpin'_\sharp$ be an $F$-inner form of $\GSpin_\sharp.$ 
Given $\vp \in \Phi_{\disc}(\GSpin_\sharp),$ for any $\sigma \in \Pi_{\vp}(\GSpin_\sharp)$ and $\sigma' \in \Pi_{\vp}(\GSpin'_\sharp),$ we have 
\begin{equation} \label{equlity two}
\mu_M(\nu, \tau \boxtimes \si, w) = \mu_{M'}(\nu, \tau' \boxtimes \si', w),
\end{equation}
where $M'$ is an $F$-inner form of $M,$ $\tau \boxtimes \si \in \Pi_{\disc}(M),$  $\tau' \boxtimes \si' \in \Pi_{\disc}(M'),$ and $\tau$ and $\tau'$ have the same $L$-parameter. 
\end{pro}
\begin{proof}
Since $\sigma_1, \sigma_2 \in \Pi_{\vp}(\GSpin_\sharp)$ are in the same restriction from an irreducible representation from $\GL_\sharp(F)$ to $\GSpin_\sharp(F),$ by \cite[Proposition 2.4]{choiy1}, we have \eqref{equlity one}.
Similarly, we note that $\si$ and $\si'$ have liftings $\ts$ and $\ts'$ in $\GL_\sharp(F)$ and $\GL'_\sharp(F),$ respectively.
Since $\tau \boxtimes \ts$ and $\tau' \boxtimes \ts'$ have the same Plancherel measures by \cite[Lemma 2.1]{ac89}, again by \cite[Proposition 2.4]{choiy1}, we have \eqref{equlity two}.
\end{proof}

Finally, we make a remark about the generic representations in $L$-packets. 
\begin{rem} \label{zeta4}
Consider the case when $\GSpin_\sharp$ is the split $\GSpin_4.$  Since $\zeta_\sharp = \mathbbm{1}$ now, 
Theorem \ref{1-1 for GSpin4} implies that we have 
\begin{equation} \label{generic bij 4}
\Pi_{\vp}({\GSpin}_4) \, \overset{1-1}{\longleftrightarrow} \, \Irr(\cS_{\vp}) \, \s \, \Irr(I^{\GSpin_4}(\ts)).
\end{equation}
Suppose that there is a generic representation in $\Pi_{\vp}({\GSpin}_4)$ with respect to a given Whittaker data for $\GSpin_4.$ 
Then, by \cite[Chapter 3]{hs11}, we can normalize  the bijection \eqref{generic bij 4} so that the trivial character 
$\mathbbm{1} \in \Irr(\cS_{\vp})$ maps to the generic representation in $\Pi_{\vp}({\GSpin}_6).$ 
\end{rem}
%

\subsection{Preservation of Local Factors via $\L_4$} \label{L-factors4}

We verify that in the case of the split form $\GSpin_4$ the map $\L_4$ preserves the local $L$-, $\epsilon$-, and $\gamma$-factors 
when these local factors are defined.  On the representation side, these factors are defined via the Langlands-Shahidi method when 
the representations are generic (or non-supercuspidal induced from generic via the Langlands classification).  On the parameter side, 
these factors are Artin factors associated to the representations of the Weil-Deligne group of $F.$ 

Since the Langlands-Shahidi method is available for generic data, we do not yet have a counterpart for the local factors for the 
non quasi-split $F$-inner forms of $\GSpin_4.$  This is the reason why we are limiting ourselves to the case of the split form below; however, 
cf. Remark \ref{inner-form-preserve} below. 

\begin{pro} \label{preserve4}
Let $\tau$ be an irreducible admissible representation of $\GL_r(F),$ $r \ge 1,$ and 
let $\si$ be an irreducible admissible representation of $\GSpin_4(F),$ which we assume to be either 
$\psi$-generic or non-supercuspidal if $r > 1.$  
Here, generic is defined with respect to a non-trivial additive character $\psi$ of $F$ in the usual way.  

Let $\vp_\tau$ be the $L$-parameter of $\tau$ via the LLC for $\GL_r(F)$ and let $\vp_\si = \L_4 (\si).$ 
Then, 
\begin{eqnarray*} 
\gamma(s, \tau \times \si, \psi) &=& \gamma(s, \vp_\tau \otimes \vp_\sigma, \psi), \\ 
L(s, \tau \times \si) &=& L(s, \vp_\tau \otimes \vp_\sigma), \\ 
\epsilon(s, \tau \times \si, \psi) &=& \epsilon(s, \vp_\tau \otimes \vp_\sigma, \psi). 
\end{eqnarray*} 
The local factors on the left hand side are those attached by Shahidi \cite[Theorem 3.5]{shahidi90annals} to the 
$\psi$-generic representations of the standard Levi subgroup $\GL_r(F) \times \GSpin_4(F)$ in $\GSpin_{2r+4}(F)$ 
and the standard representation of the dual Levi $\GL_r(\CC) \times \GSO(4,\CC),$ and extended to all non-generic non-supercuspidal 
representations via the Langlands classification and the multiplicativity of the local factors \cite[\S 9]{shahidi90annals}.  
The factors on the right hand side are Artin local factors associated to the given representations of the Weil-Deligne group of $F.$  
\end{pro}
\begin{proof}
The proof relies on two key properties of the local factors defined by Shahidi, namely that they are multiplicative (with respect to parabolic induction) 
and that they are preserved under taking irreducible constituents upon restriction.  

To be more precise, let 
$(G, M)$ and $(\tG, \tM)$ be a pair of ambient groups and standard Levi subgroups as in the Langlands-Shahidi machinery.  Assume that $G$ and 
$\tG$ satisfy \eqref{cond on G}.  
Moreover, assume that $M = G \cap \tM.$ 
If $\si$ and $\ts$ are $\psi$-generic representations of $M(F)$ and $\tM(F),$ 
respectively, such that 
\[
\sigma \hookrightarrow {\Res}_{M(F)}^{\tM(F)}(\ts),  
\]
then 
\[\gamma(s, \si, r, \psi) = \gamma(s, \ts, \tilde{r}, \psi), \\ 
\]
and similarly for the $L$- and $\epsilon$-factors.  Here, by $r$ we denote any of the irreducible constituents of the adjoint action of the complex dual of 
$M$ on the Lie algebra of the dual of the unipotent radical of the standard parabolic of $G$ having $M$ as a Levi.  Also, $\tilde{r}$ denotes 
the corresponding irreducible constituent in the complex dual of $\tG.$  Below we only need the case where $r$ and $\tilde{r}$ are standard representations. 

For a precise description of the multiplicativity property in general we refer to \cite{shahidi90ps} 
and to \cite[\S 5]{asgari-cjm} for the specific case of $\GSpin$ groups.  

Getting back to the proof, if $\tau \boxtimes \si$ is non-generic and non-supercuspidal, then it is a quotient of a standard module of 
an induced representation $\operatorname{Ind}_{P}^{\GL_4 \times \GSpin_4} (\pi),$ where $\pi$ is an essentially tempered representation 
of the standard Levi subgroup of $P.$  However, such a standard Levi only involves $\GL$ factors in our case and hence $\pi$ is generic. 
As a result, we may define the local factors associated with $\tau \boxtimes \si$ via multiplicativity. 

By the above arguments we are reduced to the case where $\tau$ and $\si$ are generic supercuspidal. The proof now follows essentially 
from the LLC for the general linear groups and \eqref{projection parameters 4} because 
\begin{eqnarray*}
\gamma(s, \tau \times \si, \psi) &=& \gamma(s, \tau \times \ts, \psi) \\
&=& \gamma(s, \vp_\tau \otimes \tilde{\vp}_{\ts}, \psi) \\
&=& \gamma(s, \vp_\tau \otimes (pr_4 \circ \tilde{\vp}_{\ts}), \psi) \\
&=& \gamma(s, \vp_\tau \otimes \vp_\sigma, \psi), \\ 
\end{eqnarray*} 
and similarly for the $L$- and $\epsilon$-factors. 
\end{proof} 
\begin{rem} \label{inner-form-preserve}
While we can not make any statement regarding the case of general (non-generic) representations of $\GSpin_4(F)$ or other inner forms due 
to the current lack of a general satisfactory theory of local factors in this generality.  Our $L$-packets would satisfy properties analogous 
to the above if such a satisfactory local theory becomes available.  By satisfactory we refer to local factors that satisfy the 
natural properties expected of them such as the so-called 
``Ten Commandments'' for the $\gamma$-factors as in \cite[Theorem 4]{lapid-rallis}. 
\end{rem}
%

\section{Local Langlands Correspondence for \texorpdfstring{$\GSpin_6$}{GSpin6} and its inner forms} \label{LLC6}

In this final section, we establish the LLC for $\GSpin_6$ and its non quasi-split $F$-inner forms.  
The method is similar to the case of $\GSpin_4$ we presented in Section \ref{LLC4},  but with somewhat 
weaker results on cardinalities of $L$-packets.

\subsection{Construction of $L$-packets of $\GSpin_6$ and its inner forms} \label{const of L-packet for GSpin6}

The discussions on restriction in Section \ref{results in rest} along with 
the description of the structure of $\GSpin_6$ in \eqref{derived of 6} imply that 
given $\sigma \in \Irr(\GSpin_6),$ there is a lifting $\ts \in \Irr({\GL}_1 \times {\GL}_4)$ such that 
\[
\sigma \hookrightarrow {\Res}_{{\GSpin}_{6}}^{\GL_1 \times \GL_4}(\ts).
\]
By the LLC for $GL_n$ \cite{ht01, he00, scholze13}, we have a unique $\tvp_{\ts} \in \Phi({\GL}_1 \times {\GL}_4)$ corresponding to the representation $\ts.$

We now define a map 
\begin{eqnarray} \label{L map for GSpin6}
{\L}_{6} : {\Irr}({\GSpin}_{6}) & \longrightarrow & \Phi({\GSpin}_{6}) \\
\sigma & \longmapsto & pr_6 \circ \tvp_{\ts}.  \nonumber
\end{eqnarray} 
Note that $\L_{4}$ does not depend on the choice of the lifting $\ts$ because 
if $\ts' \in \Irr(\GL_1 \times \GL_4)$ is another lifting, we have $\ts' \s \ts \otimes \chi$ for some quasi-character $\chi$ on
\[
({\GL}_1(F) \times {\GL}_4(F))/{\GSpin}_6(F) \s F^{\times}.  
\]
As before, 
\[
F^{\times} \s H^1(F, \CC^{\times}),
\]
where $\CC^{\times}$ is as in \eqref{ex seq L-gp for GSpin6}.
The LLC for $\GL_1 \times \GL_4$ maps
$\ts'$ to $\tvp_{\ts} \otimes \chi,$ again employing $\chi$ for both the quasi-character and 
its parameter via Local Class Field Theory. 
Since $pr_6(\tvp_{\ts} \otimes \chi) = pr_6(\tvp_{\ts})$ by \eqref{ex seq L-gp for GSpin6}, 
the map $\L_{6}$ turns out to be well-defined.

As before, $\L_6$ is surjective 
because by Labesse's Theorem in 
Section \ref{thm by Labesse}, $\vp \in \Phi(\GSpin_{6})$ 
can be lifted to some $\tvp \in \Phi(GL_1 \times \GL_4).$ 
We then obtain $\ts \in \Irr(GL_1 \times \GL_4)$ via the LLC for $\GL_1 \times \GL_4.$   
Thus, any irreducible constituent in the restriction ${\Res}_{{\GSpin}_{6}}^{\GL_1 \times \GL_4}(\ts)$ 
has the image $\vp$ via the map $\L_6.$

For each $\vp \in \Phi({\GSpin}_{6}),$ we define the $L$-packet $\Pi_{\vp}({\GSpin}_{6})$ 
as the set of all inequivalent irreducible constituents of $\ts$
\begin{equation} \label{def of L-packet for GSpin6}
\Pi_{\vp}({\GSpin}_{6}):=\Pi_{\ts}({\GSpin}_6) = 
\left\{ \sigma \hookrightarrow {\Res}_{{\GSpin}_{6}}^{\GL_1 \times \GL_4}(\ts) \right\} \slash \s,
\end{equation}
where $\ts$ is the unique member in $\Pi_{\tvp}(\GL1 \times \GL_4)$ and 
$\tvp \in \Phi(\GL_1 \times \GL_4)$ is such that $pr_{6} \circ \tvp=\vp.$ 
By the LLC for $\GL_4$ and Proposition \ref{pro for lifting}, the fiber does not depends on the choice of $\tvp.$

We define the $L$-packets for the non quasi-split inner forms similarly.  Using the group structure described in Section \ref{gp structure},
given $\sigma_6^{2,0} \in \Irr({\GSpin}_6^{2,0}),$ there is a lifting $\ts_6^{2,0} \in \Irr({\GL}_1 \times {\GL}_2(D))$ such that 
\[
\sigma_6^{2,0} \hookrightarrow {\Res}_{{\GSpin}_6^{2,0}}^{{\GL}_1 \times {\GL}_2(D)}(\ts_6^{2,0}).
\]
Again, combining the LLC for $\GL_4$ and $\GL_2(D)$ \cite{hs11},  
we have a unique $\tvp_{\ts_6^{2,0}} \in \Phi({\GL}_1 \times {\GL}_2(D))$ corresponding to the representation $\ts_6^{2,0}.$
We thus define the following map  

\begin{eqnarray} \label{L map for inner 1 GSpin6}
{\L}_6^{2,0} : {\Irr}({\GSpin}_6^{2,0}) &\longrightarrow & \Phi({\GSpin}_6^{2,0}) \\ 
\sigma^{2,0}_6 & \longmapsto & pr_6 \circ \tvp_{\ts^{2,0}_6}. \nonumber
\end{eqnarray}
Again, it follows from the LLC for $\GL_1$ and $\GL_2(D)$ that this map is well-defined and surjective.

Likewise, for the other $F$-inner form $\GSpin_6^{1,0}$ of $\GSpin_6,$ we have a well-defined and surjective map
\begin{eqnarray} \label{L map for inner 2 GSpin6}
{\L}_6^{1,0} : {\Irr}({\GSpin}_6^{1,0}) & \longrightarrow & \Phi({\GSpin}_6^{1,0}) \\
\sigma^{1,0}_6 & \longmapsto & pr_6 \circ \tvp_{\ts^{1,0}_6}. \nonumber 
\end{eqnarray}

Again, we similarly define $L$-packets 
\begin{equation} \label{def of inner 1 GSpin6}
\Pi_{\vp}({\GSpin}^{2,0}_6) = \Pi_{\ts^{2,0}_6}({\GSpin}^{2,0}_6), \quad \vp \in \Phi({\GSpin}^{2,0}_6)
\end{equation}
and 
\begin{equation} \label{def of inner 2 GSpin6}
\Pi_{\vp}({\GSpin}^{1,0}_6) = \Pi_{\ts^{1,0}_6}({\GSpin}^{1,0}_6), \quad \vp \in \Phi({\GSpin}^{1,0}_6). 
\end{equation}
These $L$-packet do not depend on the choice of $\tvp$ for similar reasons.

\subsection{Internal structure of $L$-packets of $\GSpin_6$ and its inner forms} \label{para for GSpin6}
We continue to employ the notation of Section \ref{general LLC} in this section. 
For simplicity of notation, we shall write ${\GSpin}_\flat$ for the split
${\GSpin}_{6},$ and its non quasi-split $F$-inner forms ${\GSpin}_6^{2,0}$ and ${\GSpin}_6^{1,0}.$
Likewise, we shall write ${\SL}_\flat$ and ${\GL}_\flat$ for corresponding groups in 
\eqref{derived of 6}, \eqref{derived of 6 inner 1}, and \eqref{derived of 6 inner 2} so that
we have
\begin{equation} \label{flat}
{\SL}_\flat \subset {\GSpin}_\flat \subset  {\GL}_\flat 
\end{equation}
in all cases. 
Recall from Section \ref{L-groups} that 
\begin{align*}
(\widehat{{\GSpin}_\flat})_{\ad} = {\PSO}_6(\CC) \s  {\PGL}_4(\CC), \\ 
(\widehat{{\GSpin}_\flat})_{\scn} ={\Spin}_6(\CC) \s {\SL}_4(\CC), \\
Z((\widehat{{\GSpin}_\flat})_{\scn}) = Z((\widehat{{\GSpin}_\flat})_{\scn})^{\Gamma} \s \mu_4(\CC).
\end{align*}

Let $\vp \in \Phi({\GSpin}_\flat)$ be given.
We fix a lifting $\tvp \in \Phi({\GL}_\flat)$ via the surjective map 
$\widehat{{\GL}_\flat}  \longrightarrow \widehat{{\GSpin}_\flat}$ (cf. Theorem \ref{thm by Labesse}). 
We have
\begin{align*}
S_{\vp} & \subset {\PSL}_4(\CC),  \\
S_{\vp, \scn} & \subset {\SL}_4(\CC).
\end{align*}
We then have (again by \eqref{central ext}) a central extension 
\begin{equation} 
1 \longrightarrow \widehat Z_{\vp, \scn}  \longrightarrow 
\cS_{\vp, \scn} \longrightarrow \cS_{\vp} \longrightarrow 1.
\end{equation}
Let $\zeta_{6},$ $\zeta_6^{2,0},$ and $\zeta_6^{1,0}$ be characters on $Z((\widehat{{\GSpin}_\flat})_{\scn})$ 
which respectively correspond to $\GSpin_{6},$ $\GSpin_6^{2,0},$ and $\GSpin_6^{1,0}$ via the Kottwitz isomorphism \cite[Theorem 1.2]{kot86}. 
Note that 
\[ 
\zeta_{6} = \mathbbm{1}, \quad \zeta_6^{2,0} = \sgn, \quad \mbox{ and } \quad \zeta_6^{1,0} = \hat{\zeta}_4,  
\] 
where $\sgn$ is the non-trivial character of order 2 on $\mu_4(\CC)$ and  
$\hat{\zeta}_4$ is the non-trivial character of order 4 on $\mu_4(\CC)$ whose restriction to $\mu_2$ equals $\sgn.$

\begin{thm} \label{1-1 for GSpin6}
Given an $L$-parameter $\vp \in \Phi(\GSpin_{\flat}),$ 
there is a one-one bijection 
\begin{eqnarray*}
\Pi_{\vp}({\GSpin}_{\flat}) & \, \overset{1-1}{\longleftrightarrow} \,  & \Irr(\cS_{\vp, \scn}, \zeta_{\flat}), \\
 \sigma &\mapsto & \rho_{\sigma},
\end{eqnarray*}
such that we have the following decomposition
\[
V_{\ts} ~ ~ \s \bigoplus_{\sigma \in \Pi_{\vp}({\GSpin}_{\flat})} \rho_{\sigma} \boxtimes  \si
\]
as representations of the direct product $\cS_{\vp, \scn} \times {\GSpin}_{\flat}(F),$
where $\ts \in \Pi_{\tvp}(\GL_{\flat})$ is an extension of $\si \in \Pi_{\vp}({\GSpin}_{\flat})$ to $\GL_{\flat}(F)$ 
as in Section \ref{results in rest} and $\tvp \in \Phi(\GL_{\flat})$ is a lifting of $\vp \in \Phi(\GSpin_{\flat}).$ 
Here, $\zeta_{\flat} \in \left\{ \zeta_{6}, \zeta_6^{2,0}, \zeta_6^{1,0} \right\}$ according to which inner form 
$\GSpin_{\flat}$ is.  
\end{thm}
\begin{proof}
The idea of the proof is as in the proof of Theorem \ref{1-1 for GSpin4}.
Given $\vp \in \Phi(\GSpin_{\flat}),$ we choose a lifting $\tvp \in \Phi(\GL_{\flat})$ 
and obtain the projection $\bar{\vp} \in \Phi(\SL_{\flat})$ in the following commutative diagram
\begin{equation} 
\label{projection parameters 6}
\xymatrix{
& {\quad \quad \quad \quad \quad \quad} \widehat{{\GL}_{\flat}} = {\GL}_1(\CC) \times  {\GL}_4(\CC) \ar@{->>}[d]^{{~~}pr_6}\\
W_F \times {\SL}_2(\CC) \ar@{->}[ur]^{\tvp} \ar@{->}[r]^{\vp}
\ar@{->}[dr]_{\bar{\vp}} 
& {\quad} \widehat{{\GSpin}_{\flat}} = {\GSO}_6(\CC)  \ar@{->>}^{{~~}\bar{pr}}[d] 
\\
& {\quad \quad \quad \quad \quad \quad \quad} \widehat{{\SL}_{\flat}} = {\PGL}_4(\CC) ~ .
}
\end{equation}
We then have $\ts \in \Pi_{\tvp}(\GL_{\flat})$ which is an extension of $\sigma \in \Pi_{\vp}(\GSpin_{\flat}).$
In addition to \eqref{ex seq L-gp for GSpin6}, we also have 
\[
1 \longrightarrow \CC^{\times} \times \CC^{\times} \longrightarrow {\GL}_1(\CC) \times 
{\GL}_4(\CC) \overset{\bar{pr} \circ pr_6}{\longrightarrow} {\PGL}_4(\CC) \longrightarrow 1 
\]
Considering the kernels of the projections $pr_6$ and $\bar{pr} \circ pr_6,$
we set
\[
X^{{\GSpin}_{\flat}}(\tvp) := \{ a \in H^1(W_F, \CC^{\times}) : a \tvp \s \tvp \}
\]
\[
X^{{\SL}_{\flat}}(\tvp) := \{ a \in H^1(W_F, \CC^{\times} \times \CC^{\times}) : a \tvp \s \tvp \}.
\]
Moreover, by \eqref{convenient exact sequence GSpin6} and its analogs for the two non quasi-split $F$-inner forms, we have 
\[
{\GL}_{\flat}(F) / {\GSpin}_{\flat}(F) \s F^{\times}.
\]
As an easy consequence of Galois cohomology, we also have
\[
{\GL}_{\flat}(F) / {\SL}_{\flat}(F) \s F^{\times} \times F^{\times}.
\]
Set
\[
I^{{\GSpin}_{\flat}}(\ts) := \{ \chi \in (F^{\times})^D \s ({\GL}_{\flat}(F) / {\GSpin}_{\flat}(F))^D : \ts \chi  \s \ts \}
\]
\[
I^{{\SL}_{\flat}}(\ts) := \{ \chi \in (F^{\times})^D \times (F^{\times})^D \s ({\GL}_{\flat}(F) / {\SL}_{\flat}(F))^D : \ts \chi  \s \ts \}.
\]
As in Remark \ref{rem on characters},  we often make no distinction between 
$\chi$ and $\chi \circ \det$ (respectively, $\chi \circ \Nrd$).
It follows from the definitions that 
\begin{equation} \label{subset X I 6}
X^{{\GSpin}_{\flat}}(\tvp) \subset X^{{\SL}_{\flat}}(\tvp) \quad \mbox{ and } \quad 
I^{{\GSpin}_{\flat}}(\ts) \subset I^{{\SL}_{\flat}}(\ts). 
\end{equation}
Recalling $(F^{\times})^D \s H^1(W_F, \CC^{\times})$ by the local class field theory, 
it is immediate from the LLC for $GL_n$ that
\[
X^{{\GSpin}_{\flat}}(\tvp) \s I^{{\GSpin}_{\flat}}(\ts) \quad \mbox{ and } \quad X^{{\SL}_{\flat}}(\tvp) \s I^{{\SL}_{\flat}}(\ts)
\]
as groups of characters. 
Now, with the component group notation $\cS$ of Section \ref{general LLC}, we claim that
\begin{equation} \label{crutial iso for gspin6}
I^{{\GSpin}_{\flat}}(\ts) \s \cS_{\vp}.
\end{equation}
Again, by \cite[Lemma 5.3.4]{chaoli} and above arguments, this claim follows from 
\[
\cS_{\vp} \s X^{{\GSpin}_{\flat}}(\tvp),
\]
since $\cS_{\tvp}$ is always trivial by the LLC for ${\GL}_{\flat}.$ 
Again note that $\tvp$ and $\vp$ here are respectively $\phi$ and $\phi^{\sharp}$ in \cite[Lemma 5.3.4]{chaoli}.
Thus, due to \eqref{a diagram}, \eqref{subset X I 6}, and\eqref{crutial iso for gspin6} we have
\begin{equation} \label{subset cS 6}
\cS_{\vp} \subset \cS_{\bar\vp}.
\end{equation}

Since the centralizer $C_{\bar\vp}(\widehat{{\SL}_{\flat}})$ is equal to the image of the disjoint union
\[
\coprod_{\nu \in {\Hom}(W_F, \CC^{\times})} 
\left\{ h \in {\GSO}_6(\CC) : h \vp(w)h^{-1} \vp(w)^{-1} = \nu(w) \right\}
\]
from the exact sequence
\[
1 \longrightarrow \CC^{\times} \longrightarrow {\GSO}_6(\CC) \overset{\bar{pr} }{\longrightarrow} 
\widehat{{\SL}_{\flat}} = {\PGL}_4 \longrightarrow 1,
\]
we have 
\[
S_{\vp} \subset C_{\bar\vp} = S_{\bar\vp}.
\]
So, the pre-images in $\SL_4(\CC) $ via the isogeny 
$\SL_4(\CC) \twoheadrightarrow \PGL_4(\CC)$ also satisfy
\[
S_{\vp, \scn} 
\subset
S_{\bar\vp, \scn} 
\subset {\SL}_4(\CC).
\]
This provides the inclusion of the identity components
\begin{equation} \label{subset S sc 6}
S_{\vp, \scn}^{\circ} \subset 
S_{\bar\vp, \scn}^{\circ},
\end{equation} 
and by the definition of $\widehat Z_{\vp, \scn}(G)$ of Section \ref{general LLC}, 
we again have the following surjection
\begin{equation} \label{subset Z 6}
\widehat Z_{\vp, \scn} \twoheadrightarrow \widehat Z_{\vp, \scn}({\SL}_{\flat}). 
\end{equation} 
Combining \eqref{subset cS 6}, \eqref{subset S sc 6}, and \eqref{subset Z 6}, 
we again have the commutative diagram of component groups: 
\begin{equation} \label{pre pre diagram gspin6}
\begin{CD}
1 @>>>  \widehat Z_{\vp, \scn}  @>>> \cS_{\vp, \scn} @>>> \cS_{\varphi}  @>>> 1 \\
@. @VV{\twoheadrightarrow}V @VV{\cap}V @VV{\cap}V @.\\
1 @>>>  \widehat Z_{\bar\vp, \scn}  @>>> \cS_{\bar\vp, \scn} @>>> \cS_{\bar\vp}  @>>> 1.
\end{CD} 
\end{equation}
Now apply Hiraga-Saito's homomorphism $\Lambda_{\SL_n}$ of \eqref{a diagram} 
in the case of our $\SL_{\flat},$ which we denote by $\Lambda_{\SL_4}.$ 
The restriction
\[
\Lambda_{\flat} := \Lambda_{\SL_4}\vert_{\cS_{\vp, \scn}} 
\] 
then gives the following commutative diagram: 
\begin{equation} \label{pre diagram gspin6}
\begin{CD}
1 @>>> \widehat Z_{\vp, \scn}  @>>> \cS_{\vp, \scn} @>>> \cS_{\vp}  @>>> 1 \\
@. @VV{\zeta_{\flat}}V @VV{\Lambda_{\flat}}V @VV{\cap}V @.\\
1 @>>> \CC^\times @>>> \mcA(\ts) @>>> I^{\SL_{\flat}}(\ts) @>>> 1.
\end{CD} 
\end{equation}
Note that $\zeta_{\flat}$ is identified with $\zeta_G$ in \eqref{a diagram} as the character on $\mu_4(\CC),$ 
since both are determined according to $\SL_\flat.$

Again, similar to $\mcA(\ts)$ (cf. Section \ref{results in rest}), we write $\mcA^{\GSpin_{\flat}}(\ts)$ 
for the subgroup of ${\Aut}_{\CC}(V_{\ts})$ generated by $\CC^\times$ and  
$\left\{I_{\chi} : \chi \in I^{\GSpin_{\flat}}(\ts) \right\}.$  Hence, $\mcA^{\GSpin_{\flat}}(\ts) \subset \mcA(\ts).$
By the definition of $\Lambda_{\SL_4}$ in \eqref{a diagram} and the commutative diagram \eqref{pre pre diagram gspin6}, 
it is immediate that the image of $\Lambda_{\flat}$ is 
$\mcA^{\GSpin_{\flat}}(\ts).$
We thus have the following commutative diagram: 
\begin{equation} \label{diagram gspin6}
\begin{CD}
1 @>>> \widehat Z_{\vp, \scn} @>>> \cS_{\vp, \scn} @>>> \cS_{\vp} @>>> 1 \\
@. @VV{\zeta_{\flat}}V @VV{\Lambda_{\flat}}V @VV{\s}V @.\\
1 @>>> \CC^\times @>>> \mcA^{\GSpin_{\flat}}(\ts) @>>> I^{\GSpin_{\flat}}(\ts) @>>> 1.
\end{CD} 
\end{equation}
The representation $\rho_{\sigma} \in \Irr(\cS_{\vp, \scn}, \zeta_{\flat})$ is defined by $\xi_{\si} \circ \Lambda_{\flat},$ 
where $\xi_{\si} \in \Irr(\mcA^{\GSpin_\flat}(\ts), \id)$ is the character as in the decomposition \eqref{useful decomp}.
Now, arguments of Section \ref{results in rest}, 
our construction of $L$-packets $\Pi_{\pi}(\GSpin_\flat)$ in Section \ref{const of L-packet for GSpin6} 
and diagram \eqref{diagram gspin6} above give 
\begin{equation} \label{imp bijection gspin6}
\Irr(\cS_{\vp, \scn}, \zeta_{\flat}) \, \overset{1-1}{\longleftrightarrow}
\Irr(\mcA^{\GSpin_{\flat}}(\ts), id) \, \overset{1-1}{\longleftrightarrow}
\Pi_{\vp}({\GSpin}_{\flat}).
\end{equation} 
We also have the following decomposition
\[
V_{\ts} ~ ~ \s \bigoplus_{\si \in \Pi_{\vp}({\GSpin_\flat})} \rho_{\si} \boxtimes  \si 
=
\bigoplus_{\rho \in \Irr(\cS_{\vp, \scn}, \zeta_{\flat})} \rho \boxtimes  \si_{\rho},
\]
where $\si_{\rho}$ denotes the image of $\rho$ via the bijection between $\Pi_{\vp}(\GSpin_\flat)$ 
and $\Irr(\cS_{\vp, \scn}, \zeta_{\flat}).$ Hence, the proof of Theorem \ref{1-1 for GSpin6} is complete.
\end{proof}
\begin{rem} 
As before, since $\Lambda_{\flat}$ is unique up to 
$\operatorname{Hom}(I^{\GSpin_\flat}(\ts), \CC^{\times}) \s \operatorname{Hom}(\cS_{\vp}, \CC^{\times}),$ 
the same is true for the bijection \eqref{imp bijection gspin6}.
\end{rem}
\begin{rem}
Given $\vp \in \phi(\GSpin_{6}),$ by Theorem \ref{1-1 for GSpin6}, we have a one-to-one bijection between
\[
\Pi_{\vp}({\GSpin}_{6}) \cup \Pi_{\vp}({\GSpin}_6^{2,0}) \cup 
\Pi_{\vp}({\GSpin}_6^{1,0}) \cup \Pi_{\vp}({{\GSpin}_6^{0,1}}) 
\overset{1-1}{\longleftrightarrow} {\Irr}(\cS_{\vp, \scn}),
\] 
where ${\GSpin}_6^{0,1}$ is a group isomorphic to ${\GSpin}_6^{1,0}$ as can be seen by taking the canonical 
isomorphism $D_4 \s D^{op}_4,$ as discussed in \eqref{derived of 6 inner 2}.
\end{rem}

\subsection{$L$-packet Sizes for $\GSpin_6$ and Its Inner Forms} \label{size of L-pacekt gspin6}
Just as in the case of $\GSpin_4,$ we have the following possible cardinalities for the $L$-packets 
of $\GSpin_6$ and its inner forms. 
\begin{pro} \label{pro size 6}
Let $\Pi_{\vp}(\GSpin_\flat)$ be an $L$-packet associated to $\vp \in \Phi(\GSpin_\flat).$  
Then we have
\[
\left| \Pi_{\vp}({\GSpin}_\flat) \right| 
\; \Big{|} \; 
\left| F^{\times}/(F^{\times})^2 \right|,
\]
which  yields the possible cardinalities  
\[
\left| \Pi_{\vp}({\GSpin}_\flat) \right| 
= \left\{ 
\begin{array}{l l}
    1, ~ 2, ~ 4, & \: \text{if} ~ p \neq 2, \\

    1, ~ 2, ~ 4, ~ 8, & \: \text{if} ~ p= 2. \\
  \end{array}
  \right.
\]
\end{pro}
\begin{proof}
We proceed similarly as in the proof of Proposition \ref{pro size 4}, again making use of Galois Cohomology.  
The exact sequence of algebraic groups
\[
1 \longrightarrow Z({\GSpin}_\flat) \longrightarrow Z({\GL}_\flat) \times {\GSpin}_\flat \longrightarrow {\GL}_\flat 
\longrightarrow 1
\]
gives a long exact sequence
\[
\cdots \longrightarrow (F^{\times} \times F^{\times}) \times {\GSpin}_\flat(F) \longrightarrow {\GL}_\flat(F) 
\longrightarrow H^1(F, Z({\GSpin}_\flat)) \longrightarrow H^1(F, Z({\GL}_\flat) \times {\GSpin}_\flat) \longrightarrow 1.
\]
Since $H^1(F, Z({\GL}_\flat) \times {\GSpin}_\flat) = 1$ by Lemma \ref{GC for GSpin} and \cite[Lemmas 2.8]{pr94},  
we have 
\[
{\GL}_\flat(F) / \left( (F^{\times} \times F^{\times}) \times {\GSpin}_\flat(F) \right) \hookrightarrow  H^1(F, Z({\GSpin}_\flat)).
\]
Also, the exact sequence
\[
1 \longrightarrow Z({\GSpin}_\flat)^{\circ} \longrightarrow Z({\GSpin}_\flat) \longrightarrow \pi_0(Z({\GSpin}_\flat)) \longrightarrow 1, 
\]
we have 
\[
H^1(F, Z({\GSpin}_\flat)) \hookrightarrow H^1(F, \pi_0( Z({\GSpin}_\flat))), 
\]
since, by \cite[Proposition 2.3]{ash06duke}, $Z({\GSpin}_\flat)^{\circ} \s \GL_1$  and $H^1(F, Z({\GSpin}_\flat)^{\circ})=1.$
Combining the above with \eqref{centers}, we have
\begin{equation} \label{coker imbeding 6} 
{\GL}_\flat(F) / \left( (F^{\times} \times F^{\times}) \times {\GSpin}_\flat(F) \right) 
\hookrightarrow  H^1(F, \pi_0( Z({\GSpin}_\flat))) \s H^1(F, \ZZ/2\ZZ) \s F^{\times}/(F^{\times})^2.
\end{equation}
On the other hand, we know from Section \ref{results in rest} that 
\[
\Pi_{\vp}({\GSpin}_\flat) \overset{1-1}{\longleftrightarrow} 
{\GL}_\flat(F) / {\GL}_\flat(F)_{\ts}
\hookrightarrow 
{\GL}_\flat(F) / \left( (F^{\times} \times F^{\times}) \times {\GSpin}_\flat(F) \right),
\]
where $\ts \in \Irr(\GL_\flat)$ corresponds to a lifting $\tvp \in \Phi(\GL_\flat),$ via the LLC for $GL_n$ and its inner forms, 
of $\vp \in \Phi(\GSpin_\flat).$ 
This completes the proof. 
\end{proof}

Next, we give a description of the group $I^{\GSpin_6}(\ts)$ for $\ts \in \Irr(\GL_1 \times \GL_4).$  
Unlike the case of $\GSpin_4$ we are unable to give a case by case classification (cf. Remark \ref{rem for GSpin6}).

Given $\ts \in \Irr(\GL_1 \times \GL_4),$ we set $\ts = \teta \boxtimes \ts_0$ with $\teta \in (\GL_1(F))^D, \ts_0 \in \Irr(\GL_4).$
By Remark \ref{rem on characters}, we note that $\chi \in ({\GL}_{\flat}(F))^D$ is decomposed into 
$\tchi_1 \boxtimes \tchi_2,$ where $\tchi_1 \in (\GL_1(F))^D$ and $\tchi_2 \in ({\GL}_{\flat}(F))^D.$ 
Moreover, we identify $\tchi_2 \in ({\GL}_{4}(F))^D$ and $\tchi_2 \circ \det,$ since any character on 
$\GL_n(F)$ is of the form $\tchi \circ \det$ for some character $\tchi$ on $F^{\times}.$

\begin{lm} \label{lemma form of characters for GSpin6}
Any $\chi \in I^{\GSpin_6}(\ts)$ is of the form 
\[
\tchi^{-2} \boxtimes \tchi,
\]
for some $\tchi \in (F^{\times})^D.$  
\end{lm}
\begin{proof}
Since $\chi = \tchi_1 \boxtimes \tchi_2$ as above and $\chi$ is trivial on $\GSpin_6(F),$ by the structure of 
$\GSpin_6(F)$ in \eqref{convenient description of GSpin6(F)}, we have 
\[
\chi((g_1, g_2))=\tchi_1(g_1) \boxtimes \tchi_2 (\det g_2)=\tchi_1(g_1) \tchi_2 (\det g_2) =1
\]
for all $(g_1,g_2) \in \GL_1(F) \times \GL_4(F)$ with $(g_1)^2 = \det g_2.$ 
Since the determinant map $\det: \GL_4(F) \rightarrow F^{\times}$ is surjective, 
$\tchi_1 (\tchi_2)^{2}(x)$ must be trivial for all $x \in F^{\times}.$ Thus, we have $\tchi_1=(\tchi_2)^{-2}.$ 
\end{proof}
\begin{pro} \label{form of characters for I(GSpin6)}
We have
\[ 
I^{\GSpin_6}(\ts) = 
\{
\chi \in I^{\SL_4}(\ts_0) : \chi^2 = \mathbbm{1} 
\}.
\] 
\end{pro}
\begin{proof}
Since $\ts = \teta \boxtimes \ts_0$ with $\teta \in (\GL_1(F))^D$ and $\ts_0 \in \Irr(\GL_4),$ 
by Lemma \ref{lemma form of characters for GSpin6}, we have
\begin{equation*} \label{iff for characters GSpin6}
\ts \chi \s \ts 
\;  \Longleftrightarrow \;
\teta \tchi^{-2} \boxtimes \ts_0 \tchi \s \teta \boxtimes \ts_0
\;  \Longleftrightarrow \;
\ts_0 \tchi \s \ts_0  \text{ and }  \tchi^{2}=\mathbbm{1}.
\end{equation*}
This completes the proof.
\end{proof}
\begin{rem} \label{rem for GSpin6}
Propositions \ref{pro size 6} and \ref{form of characters for I(GSpin6)} imply that $I^{\GSpin_6}(\ts)$ is of the form 
$(\ZZ/2\ZZ)^r$ with $r=0, 1, 2$ if $p \neq 2$ and with $r=0, 1, 2, 3$ if $p = 2.$
Unlike the case of $\SL_2$ (see \cite[Porposition 6.3]{gtsp10}),  a full classification of irreducible $L$-parameters in 
$\Phi(\SL_4)$ is not currently available.  Thus, unlike the case of $\GSpin_4$ in Section \ref{size of L-pacekt gspin4}  
(cf. Remarks \ref{cor for I(GSpin4)} and \ref{rem 2 for I(GSpin4)}), 
we do not classify the group $I^{\GSpin_6}(\ts)$ case by case, 
nor discuss the group structure of $\cS_{\vp, \scn},$ and accordingly, the description of all the sizes of $L$-packets for $\GSpin_{6},$ $\GSpin^{2,0}_6,$ and $\GSpin^{1,0}_6$ is not available here. 
\end{rem}

\subsection{Properties of $\L$-maps for $\GSpin_{6}$ and its inner forms} \label{properties for gspin6}

The $\L$-maps defined in Section \ref{const of L-packet for GSpin6} again satisfy 
some natural properties similar to the case of $\GSpin_4.$  We now verify those properties. 
In what follows, let $\L_{\flat} \in \left\{\L_6, \L_6^{2,0}, \L_4^{1,0} \right\}$ as the case may be. 

\begin{pro} \label{discreteness for GSpin6}
A representation $\sigma_{\flat} \in \Irr({\GSpin}_\flat)$ is essentially square integrable 
if and only if its $L$-parameter $\vp_{\sigma_{\flat}} := \L_{\flat}(\sigma_{\flat})$ does not factor 
through any proper Levi subgroup of ${\GSO}_6(\CC).$
\end{pro}

\begin{proof}
The proof is similar to that of Proposition \ref{discreteness for gspin4} so we omit the details. 
\end{proof}
\begin{rem}
We similarly have that a given $\sigma_{\flat} \in \Irr({\GSpin}_\flat)$ is tempered if and only 
if the image of its $L$-parameter $\vp_{\sigma_{\flat}} := \L_{\flat}(\sigma_{\flat})$ in ${\GSO}_6(\CC)$ is bounded. 
\end{rem}

Again because the restriction of representations preserves the intertwining operator 
and the Plancherel measure (see \cite[Section 2.2]{choiy1}), our construction of the 
$L$-packets in Section \ref{const of L-packet for GSpin6}, gives the following result. 

\begin{pro}
Let $\vp \in \Phi_{\disc}(\GSpin_\flat)$ be given.  For any $\sigma_1, \sigma_2 \in \Pi_{\vp}(\GSpin_\flat),$ 
we have the equality of the Plancherel measures 
\[
\mu_M(\nu, \tau \boxtimes \si_1, w) = \mu_M(\nu, \tau \boxtimes \si_2, w),
\]
where $M$ is an $F$-Levi subgroup of an $F$-inner form of $\GSpin_{2n}$ of the form of the product of 
$\GSpin_\flat$ and copies of $F$-inner forms of $\GL_{m_i},$ $\tau \boxtimes \si_1, \tau \boxtimes \si_2 \in \Pi_{\disc}(M),$ 
$\nu \in \mathfrak{a}^{*}_{M, \CC},$ and $w \in W_M$ with $^wM = M.$
Further, it is a consequence of the equality of the Plancherel measures that the Plancherel measure is also preserved 
between $F$-inner forms in the following sense.  Let $\GSpin'_\flat$ be an $F$-inner form of $\GSpin_\flat.$ 
Given $\vp \in \Phi_{\disc}(\GSpin_\flat),$ for any 
$\sigma \in \Pi_{\vp}(\GSpin_\flat)$ and $\sigma' \in \Pi_{\vp}(\GSpin'_\flat)$ we have 
\[
\mu_M(\nu, \tau \boxtimes \si, w) = \mu_{M'}(\nu, \tau' \boxtimes \si', w),
\]
where $M'$ is an $F$-inner form of $M,$ $\tau \boxtimes \si \in \Pi_{\disc}(M),$ $\tau' \boxtimes \si' \in \Pi_{\disc}(M'),$ and $\tau$ and $\tau'$ have the same $L$-parameter. 
\end{pro} 
\begin{proof}
The proof is similar to that of Proposition \ref{pm for gspin4} so we omit the details. 
\end{proof}

\begin{rem} \label{zeta6}
Similar to Remark \ref{zeta4}, in the case where $\GSpin_\flat$ is the split group $\GSpin_6,$ 
we have $\zeta_\flat= \mathbbm{1}$ and Theorem \ref{1-1 for GSpin6} implies
\begin{equation} \label{generic bij 6}
\Pi_{\vp}({\GSpin}_6) \, \overset{1-1}{\longleftrightarrow} \, \Irr(\cS_{\vp}(\widehat{{\GSpin}_6})) \, \s \, \Irr(I^{\GSpin_6}(\ts)). 
\end{equation}
Suppose that there is a generic representation in $\Pi_{\vp}({\GSpin}_6)$ with respect to a given Whittaker data for $\GSpin_6.$ 
Then, by \cite[Chapter 3]{hs11}, we can normalize  the bijection \eqref{generic bij 6} above, so that the trivial character 
$\mathbbm{1} \in \Irr(\cS_{\vp})$ maps to the generic representation in $\Pi_{\vp}({\GSpin}_6).$ 
\end{rem}
%

\subsection{Preservation of Local Factors via $\L_6$} \label{L-factors6}

As in Section \ref{L-factors4}, we can again verify that in the case of the split form $\GSpin_6$ the map $\L_6$ preserves 
the local $L$-, $\epsilon$-, and $\gamma$-factors when these local factors are defined.  
We recall that on the representation side, these factors are defined via the Langlands-Shahidi method when 
the representations are generic (or non-supercuspidal induced from generic via the Langlands classification) 
and on the parameter side, they are Artin factors associated to the representations of the Weil-Deligne group of $F.$ 

Since the Langlands-Shahidi method is available for generic data, we do not yet have a counterpart for the local factors for the 
non quasi-split $F$-inner forms of $\GSpin_6.$  This is the reason why we are limiting ourselves to the case of the split form below; however, 
a remark similar to Remark \ref{inner-form-preserve} applies again. 

\begin{pro} \label{preserve6}
Let $\tau$ be an irreducible admissible representation of $\GL_r(F),$ $r \ge 1,$ and 
let $\si$ be an irreducible admissible representation of $\GSpin_6(F),$ which we assume to be either 
$\psi$-generic or non-supercuspidal if $r > 1.$  
Here, generic is defined with respect to a non-trivial additive character $\psi$ of $F$ in the usual way.  

Let $\vp_\tau$ be the $L$-parameter of $\tau$ via the LLC for $\GL_r(F)$ and let $\vp_\si = \L_6 (\si).$ 
Then, 
\begin{eqnarray*} 
\gamma(s, \tau \times \si, \psi) &=& \gamma(s, \vp_\tau \otimes \vp_\sigma, \psi), \\ 
L(s, \tau \times \si) &=& L(s, \vp_\tau \otimes \vp_\sigma), \\ 
\epsilon(s, \tau \times \si, \psi) &=& \epsilon(s, \vp_\tau \otimes \vp_\sigma, \psi). 
\end{eqnarray*} 
The local factors on the left hand side are those attached by Shahidi \cite[Theorem 3.5]{shahidi90annals} to the 
$\psi$-generic representations of the standard Levi subgroup $\GL_r(F) \times \GSpin_6(F)$ in $\GSpin_{2r+6}(F)$ 
and the standard representation of the dual Levi $\GL_r(\CC) \times \GSO(6,\CC),$ and extended to all non-generic non-supercuspidal 
representations via the Langlands classification and the multiplicativity of the local factors \cite[\S 9]{shahidi90annals}.  
The factors on the right hand side are Artin local factors associated to the given representations of the Weil-Deligne group of $F.$  
\end{pro}
\begin{proof}
The proof is similar to that of Proposition \ref{preserve4} and we will not repeat it. 
\end{proof} 
%



\begin{thebibliography}{KMSW14}

\bibitem[Art06]{art06}
James Arthur, \emph{A note on {$L$}-packets}, Pure Appl. Math. Q. \textbf{2}
  (2006), no.~1, Special Issue: In honor of John H. Coates. Part 1, 199--217.
  \MR{2217572 (2006k:22014)}

\bibitem[Art13]{art12}
\bysame, \emph{The endoscopic classification of representations}, American
  Mathematical Society Colloquium Publications, vol.~61, American Mathematical
  Society, Providence, RI, 2013, Orthogonal and symplectic groups. \MR{3135650}

\bibitem[AC89]{ac89}
James Arthur and Laurent Clozel, \emph{Simple algebras, base change, and the
  advanced theory of the trace formula}, Annals of Mathematics Studies, vol.
  120, Princeton University Press, Princeton, NJ, 1989. \MR{1007299
  (90m:22041)}

\bibitem[Asg02]{asgari-cjm}
Mahdi Asgari, \emph{Local {$L$}-functions for split spinor groups}, Canad. J.
  Math. \textbf{54} (2002), no.~4, 673--693. \MR{1913914 (2003i:11062)}

\bibitem[ACS1]{acs-local}
Mahdi Asgari, James~W. Cogdell, and Freydoon Shahidi, 
\emph{Local Transfer and Reducibility of Induced Representations of p-adic Groups of Classical Type}. 
Contemp. Math., Amer. Math. Soc., Providence, RI, to appear.  

\bibitem[ACS2]{acsh}
Mahdi Asgari, James~W. Cogdell, and Freydoon Shahidi, 
\emph{Rankin-{S}elberg  $\rm L$-functions for {$\rm{GSpin \times GL}$} {G}roups}. 
In preparation.

\bibitem[AS06]{ash06duke}
Mahdi Asgari and Freydoon Shahidi, \emph{Generic transfer for general spin
  groups}, Duke Math. J. \textbf{132} (2006), no.~1, 137--190. \MR{2219256
  (2007d:11055a)}

\bibitem[AS14]{ash14manuscripta}
\bysame, \emph{Image of functoriality for general spin groups}, Manuscripta
  Math. \textbf{144} (2014), no.~3-4, 609--638. \MR{3227529}
 
 
  \bibitem[ABPS14]{abps13}
Anne-Marie Aubert, Paul Baum, Roger Plymen, and Maarten Solleveld, \emph{The
  local {L}anglands correspondence for inner forms of {$SL_n$}},
 arXiv:1305.2638v3 [math.RT] (2014).

\bibitem[Bad08]{ba08}
Alexandru~Ioan Badulescu, \emph{Global {J}acquet-{L}anglands correspondence,
  multiplicity one and classification of automorphic representations}, Invent.
  Math. \textbf{172} (2008), no.~2, 383--438, With an appendix by Neven Grbac.
  \MR{2390289 (2009b:22016)}

\bibitem[Bor79]{bo79}
Armand Borel, \emph{Automorphic {$L$}-functions}, Proc. Sympos. Pure Math.,
  XXXIII, Amer. Math. Soc., Providence, R.I., 1979, pp.~27--61. \MR{546608
  (81m:10056)}

\bibitem[BH06]{bh06}
Colin~J. Bushnell and Guy Henniart, \emph{The local {L}anglands conjecture for
  {$\rm GL(2)$}}, Grundlehren der Mathematischen Wissenschaften [Fundamental
  Principles of Mathematical Sciences], vol. 335, Springer-Verlag, Berlin,
  2006. \MR{2234120 (2007m:22013)}

\bibitem[CG14]{gan-chan} 
Ping-Shun Chan and Wee~Teck Gan,  \emph{The local {L}anglands conjecture for
  {$\rm GSp(4)$}, {III}: {S}tability and twisted endoscopy}, 
  J. Number Theory, \textbf{146} (2015), 69--133. \MR{3267112}
    
\bibitem[CL14]{chaoli}
Kuok~Fai Chao and Wen-Wei Li, \emph{Dual {$R$}-groups of the inner forms of
  {${\rm SL}(N)$}}, Pacific J. Math. \textbf{267} (2014), no.~1, 35--90.
  \MR{3163476}
 
\bibitem[Cho14a]{choiy1}
Kwangho Choiy, \emph{Transfer of {P}lancherel measures for unitary supercuspidal
  representations between {$p$}-adic inner forms}, Canad. J. Math. \textbf{66}
  (2014), no.~3, 566--595. \MR{3194161}

\bibitem[Cho14b]{choiymulti}
\bysame, \emph{On multiplicity in restriction for $p$-adic groups}, arXiv:1306.6118v5 [math.NT].
 
\bibitem[Cho14c]{choiysp11}
\bysame, \emph{The local Langlands conjecture for the {$p$}-adic inner form of {$\rm Sp_4$}},  Int. Math. Res. Not. IMRN, to appear.

\bibitem[CG15]{chgo12}
Kwangho Choiy and David Goldberg, \emph{Transfer of {$R$}-groups between
  {$p$}-adic inner forms of {$\rm SL_n$}}, Manuscripta Math. \textbf{146} (2015),
  no.~1-2, 125--152. \MR{3294420}

\bibitem[DKV84]{dkv}
Pierre Deligne, David Kazhdan, and Marie-France Vign{\'e}ras,
  \emph{Repr\'esentations des alg\`ebres centrales simples {$p$}-adiques},
  Travaux en Cours, Hermann, Paris, 1984, pp.~33--117. \MR{771672 (86h:11044)}

\bibitem[GT10]{gtsp10}
Wee~Teck Gan and Shuichiro Takeda, \emph{The local {L}anglands conjecture for
  {S}p(4)}, Int. Math. Res. Not. IMRN (2010), no.~15, 2987--3038. \MR{2673717
  (2011g:22027)}

\bibitem[GT11]{gt}
\bysame, \emph{The local {L}anglands conjecture for {${\rm GSp}(4)$}}, Ann. of
  Math. (2) \textbf{173} (2011), no.~3, 1841--1882. \MR{2800725}

\bibitem[GT14]{gtan12}
Wee~Teck Gan and Welly Tantono, \emph{The local {L}anglands conjecture for
  {$\rm GSp(4)$}, {II}: {T}he case of inner forms}, Amer. J. Math. \textbf{136}
  (2014), no.~3, 761--805. \MR{3214276}

\bibitem[GK82]{gk82}
Stephen~S. Gelbart and Anthony~W. Knapp, \emph{{$L$}-indistinguishability and
  {$R$} groups for the special linear group}, Adv. in Math. \textbf{43} (1982),
  no.~2, 101--121. \MR{644669 (83j:22009)}

\bibitem[GP92]{grossprasad92}
Benedict~H. Gross and Dipendra Prasad, \emph{On the decomposition of a
  representation of {${\rm SO}_n$} when restricted to {${\rm SO}_{n-1}$}},
  Canad. J. Math. \textbf{44} (1992), no.~5, 974--1002. \MR{1186476
  (93j:22031)}

\bibitem[HT01]{ht01}
Michael Harris and Richard Taylor, \emph{The geometry and cohomology of some
  simple {S}himura varieties}, Annals of Mathematics Studies, vol. 151,
  Princeton University Press, Princeton, NJ, 2001, With an appendix by Vladimir
  G. Berkovich. \MR{1876802 (2002m:11050)}

\bibitem[Hen80]{he80}
Guy Henniart, \emph{Repr\'esentations du groupe de {W}eil d'un corps local},
  Enseign. Math. (2) \textbf{26} (1980), no.~1-2, 155--172. \MR{590513
  (81j:12012)}

\bibitem[Hen00]{he00}
\bysame, \emph{Une preuve simple des conjectures de {L}anglands pour {${\rm
  GL}(n)$} sur un corps {$p$}-adique}, Invent. Math. \textbf{139} (2000),
  no.~2, 439--455. \MR{1738446 (2001e:11052)}

\bibitem[HS12]{hs11}
Kaoru Hiraga and Hiroshi Saito, \emph{On {$L$}-packets for inner forms of
  {$SL_n$}}, Mem. Amer. Math. Soc. \textbf{215} (2012), no.~1013, vi+97.
  \MR{2918491}

\bibitem[HS75]{hosh75}
Roger Howe and Allan Silberger, \emph{Why any unitary principal series
  representation of {${\rm SL}_n$} over a {$p$}-adic field decomposed simply},
  Bull. Amer. Math. Soc. \textbf{81} (1975), 599--601. \MR{0369623 (51 \#5855)}

\bibitem[JL70]{jl}
Herv{\'e} Jacquet and Robert~P. Langlands, \emph{Automorphic forms on {${\rm
  GL}(2)$}}, Lecture Notes in Mathematics, Vol. 114, Springer-Verlag, Berlin,
  1970. \MR{0401654 (53 \#5481)}

\bibitem[KMSW14]{kmsw14}
Tasho Kaletha, Alberto Minguez, Sug~Woo Shin, and Paul-James White,
  \emph{Endoscopic classification of representations: Inner forms of unitary
  groups}, arXiv:1409.3731v2 [math.NT] (2014).

\bibitem[Koc77]{koch77}
Helmut Koch, \emph{Classification of the primitive representations of the
  {G}alois group of local fields}, Invent. Math. \textbf{40} (1977), no.~2,
  195--216. \MR{0450244 (56 \#8540)}

\bibitem[Kot84]{kot84}
Robert~E. Kottwitz, \emph{Stable trace formula: cuspidal tempered terms}, Duke
  Math. J. \textbf{51} (1984), no.~3, 611--650. \MR{757954 (85m:11080)}

\bibitem[Kot86]{kot86}
\bysame, \emph{Stable trace formula: elliptic singular terms}, Math. Ann.
  \textbf{275} (1986), no.~3, 365--399. \MR{858284 (88d:22027)}

\bibitem[Kot97]{kot97}
\bysame, \emph{Isocrystals with additional structure. {II}}, Compositio Math.
  \textbf{109} (1997), no.~3, 255--339. \MR{1485921 (99e:20061)}

\bibitem[Lab85]{la85}
Jean-Pierre Labesse, \emph{Cohomologie, {$L$}-groupes et fonctorialit\'e},
  Compositio Math. \textbf{55} (1985), no.~2, 163--184. \MR{795713 (86j:11117)}

\bibitem[LL79]{ll79}
Jean-Pierre Labesse and Robert~P. Langlands, \emph{{$L$}-indistinguishability
  for {${\rm SL}(2)$}}, Canad. J. Math. \textbf{31} (1979), no.~4, 726--785.
  \MR{540902 (81b:22017)}

\bibitem[LR05]{lapid-rallis}
Erez~M. Lapid and Stephen Rallis, \emph{On the local factors of representations
  of classical groups}, Automorphic representations, {$L$}-functions and
  applications: progress and prospects, Ohio State Univ. Math. Res. Inst.
  Publ., vol.~11, de Gruyter, Berlin, 2005, pp.~309--359. \MR{2192828
  (2006j:11071)}

\bibitem[Mok15]{mok13}
Chung~Pang Mok, \emph{Endoscopic classification of representations of
  quasi-split unitary groups}, Mem. Amer. Math. Soc. \textbf{235} (2015),
  no.~1108, vi+248. \MR{3338302}

\bibitem[PR94]{pr94}
Vladimir Platonov and Andrei Rapinchuk, \emph{Algebraic groups and number
  theory}, Pure and Applied Mathematics, vol. 139, Academic Press Inc., Boston,
  MA, 1994. \MR{1278263 (95b:11039)}

\bibitem[Rog83]{rog83}
Jonathan~D. Rogawski, \emph{Representations of {${\rm GL}(n)$} and division
  algebras over a {$p$}-adic field}, Duke Math. J. \textbf{50} (1983), no.~1,
  161--196. \MR{700135 (84j:12018)}

\bibitem[Rog90]{rog90}
\bysame, \emph{Automorphic representations of unitary groups in three
  variables}, Annals of Mathematics Studies, vol. 123, Princeton University
  Press, Princeton, NJ, 1990. \MR{1081540 (91k:22037)}

\bibitem[Sat71]{sa71}
Ichir{\^o} Satake, \emph{Classification theory of semi-simple algebraic
  groups}, Marcel Dekker Inc., New York, 1971, With an appendix by M. Sugiura,
  Notes prepared by Doris Schattschneider, Lecture Notes in Pure and Applied
  Mathematics, 3. \MR{0316588 (47 \#5135)}

\bibitem[Sha90a]{shahidi90ps}
Freydoon Shahidi, \emph{On multiplicativity of local factors}, Festschrift in
  honor of {I}. {I}. {P}iatetski-{S}hapiro on the occasion of his sixtieth
  birthday, {P}art {II} ({R}amat {A}viv, 1989), Israel Math. Conf. Proc.,
  vol.~3, Weizmann, Jerusalem, 1990, pp.~279--289. \MR{1159120 (93e:11144)}

\bibitem[Sha90b]{shahidi90annals}
\bysame, \emph{A proof of {L}anglands' conjecture on {P}lancherel measures;
  complementary series for {$p$}-adic groups}, Ann. of Math. (2) \textbf{132}
  (1990), no.~2, 273--330. \MR{1070599 (91m:11095)}

\bibitem[Sch13]{scholze13}
Peter Scholze, \emph{The local {L}anglands correspondence for {$\rm{GL}_n$}
  over {$p$}-adic fields}, Invent. Math. \textbf{192} (2013), no.~3, 663--715.
  \MR{3049932}

\bibitem[Tad92]{tad92}
Marko Tadi{\'c}, \emph{Notes on representations of non-{A}rchimedean {${\rm
  SL}(n)$}}, Pacific J. Math. \textbf{152} (1992), no.~2, 375--396. \MR{1141803
  (92k:22029)}

\bibitem[Wei74]{weil74}
Andr{\'e} Weil, \emph{Exercices dyadiques}, Invent. Math. \textbf{27} (1974),
  1--22. \MR{0379445 (52 \#350)}

\end{thebibliography}
\end{document}